\documentclass[reqno]{amsart}
\usepackage{amsmath}
\usepackage{amssymb}
\usepackage{mathtools}
\usepackage{amsaddr}
\usepackage{enumerate}
\usepackage{tikz}
\usetikzlibrary{matrix}
\usepackage{hyperref}
\usepackage{graphicx}  
\usepackage{longtable}
\usepackage{caption}
\usepackage{subcaption}
\usepackage{array}
\newcolumntype{P}[1]{>{\centering\arraybackslash}p{#1}}
\usepackage{adjustbox}
\usepackage{multirow}
\usepackage{lscape}
\usepackage{afterpage}
\title{Characteristic properties of $\sigma$-A-nuclei of a quasigroup}
\author[Dimpy Chauhan, Indivar Gupta, Rashmi Verma]
{\textbf{Dimpy Chauhan$^{1,\ast}$}, \textbf{Indivar Gupta$^2$}, \textbf{Rashmi Verma$^3$}}

\address{$^1$ Department of Mathematics, University of Delhi, Delhi-110007, India} 
\address{$^2$ SAG, Metcalfe house, DRDO, Delhi-110054, India} 
\address{$^3$ Mata Sundri College for Women, University of Delhi, Delhi-110002 India}

\thanks{$^\ast$Corresponding Author}
\email{dimpychauhan21jan@gmail.com (D.Chauhan)}
\email{indivar\_gupta@yahoo.com (I.Gupta)}
\email{rashmiv710@gmail.com (R.Verma)}

\theoremstyle{definition}
\newtheorem{definition}{Definition}

\theoremstyle{theorem}
\newtheorem{theorem}{Theorem}
\newtheorem{corollary}{Corollary}[theorem]

\newtheorem{remark}{Remark}
\begin{document}
	\maketitle
	
	\begin{abstract}
	In this paper, we investigate the properties of $\sigma$-A-nuclei of a quasigroup including relations between them and relations between their respective component sets, where $\sigma \in S_3$. We also find connections between components of $\sigma$-A-nuclei of a quasigroup and components of $\sigma$-A-nuclei of the isostrophic images. 
	Further, we investigate the properties of various inverse quasigroups using the derived connections. These properties will not only make the study of $\sigma$-A-nuclei of a quasigroup simple but also reduce the time required in the computation of $\sigma$-A-nuclei of a quasigroup for different values of $\sigma$. 
	\end{abstract}
	
\section{Introduction}
Garrison \cite{Garrison} introduced the notion of quasigroup nuclei in 1940. It measures how far a quasigroup is from a group by measuring the near-associativity of the quasigroup. It has been shown in \cite{Pflugfelder} that if a quasigroup $(Q,\cdot)$ has a non-trivial left/ right/ middle nucleus then the quasigroup $(Q,\cdot)$ has a left identity/ a right identity/ an identity element. A detailed historical background of A-nuclei of a quasigroup has been discussed by V.A. Shcherbacov in \cite{Shcherbacov&2011,Shcherbacov&2017}. Belousov \cite{Belousov&58} generalised Garrison's nuclei to proper quasigroups by introducing the concept of left, right, and middle regular permutations, which can be used to measure the quasigroups' near-associativity in the absence of a left or right unit.
Several authors have investigated connections between quasigroup nuclei and groups of regular permutations of quasigroups \cite{Belousov&58, Belousov&67, Belousov, Kepka}.
The A-nuclei (autotopy nuclei) of a quasigroup are then identified as groups of left, right, and middle regular permutations, and have been studied extensively by Belousov, Kepka, Keedwell, Anthony, Shcherbacov, etc \cite{Belousov&58, Kepka, Keedwell&Anthony, Shcherbacov&2011}.
It is worth noting that if the quasigroup is a loop, then the non-trivial components of the loop's left/right A-nuclei coincide with left/right translations by elements of the left/right nuclei, and hence the notion of A-nuclei may be considered as a generalization of a quasigroup's nuclei.
The connections between components of the A-nuclei of a quasigroup and components of the A-nuclei of its isostrophic images have been discussed in \cite{Shcherbacov&2017}. The reader may refer \cite{Shcherbacov&2011,Shcherbacov&2017} for a detailed survey on A-nuclei of a quasigroup.

The concept of left, right and middle $\sigma$-A-nuclei of a quasigroup, as a generalization of left, right and middle A-nuclei respectively, have been introduced in \cite{Chauhan22}. 
In this paper we characterize the inverse sets of $\sigma$-A nuclei and their respective components. Further, characterization of the products of $\sigma$-A-nuclei and $\tau$-A-nuclei
of a quasigroup and their respective component sets have been discussed. We also find connections between $\sigma$-A-nuclei of a quasigroup and $\sigma$-A-nuclei of its isostrophic images, as well as between their respective component sets. Further, we use the connections of $\sigma$-A-nucleus of a quasigroup and its isostrophic images to investigate properties of $\sigma$-A-nuclei of various inverse quasigroups.

The paper has the following structure: Section 2 presents some basic definitions, notations and results required for the subsequent sections. In Section 3, we define $\sigma$-A-nuclei of a quasigroup and investigate some of their properties. Further, we derive connections between components of the $\sigma$-A-nuclei of a quasigroup and its isostrophic images. Section 4 discusses properties of various inverse quasigroups. Conclusions are finally drawn in Section 5. 

\section{Preliminaries}

In this section, we present some definitions and notations required for our study in this paper \cite{Belousov, Pflugfelder, Shcherbacov&2003, Shcherbacov&2009, Shcherbacov&2011, Shcherbacov&2017}.

\begin{definition} \label{qg}
	A groupoid $(Q,\ast)$ is called a \textit{quasigroup} if there exist unique solutions $x,y \in Q$ of the equations $x\ast a =b $ and $a\ast y=b $ for all ordered pair $(a,b)\in Q^2$. 
\end{definition}

In other words, a groupoid $(Q,\ast)$ is called a quasigroup if in the equality $x \ast y=z$, knowing any two elements from $x,y,z$ uniquely specifies the remaining one element.

In view of the above definition, given a quasigroup $(Q,\ast)$ it is possible to associate five other operators $\ast^{(1\,2)}, \ast^{(1\,3)}, \ast^{(2\,3)}, \ast^{(1\,2\,3)}$ and $\ast^{(1\,3\,2)}$, known as \textit{parastrophes} of quasigroup $(Q,\ast)$ as follows:\\
$x \ast y = z
\iff y \ast^{(1\,2)} x =z 
\iff z \ast ^{(1\,3)} y= x 
\iff x \ast^{(2\,3)} z = y 
\iff y \ast^{(1\,2\,3)} z = x\\
\iff z \ast^{(1\,3\,2)}x=y. $

Note that the quasigroup $(Q,\ast^{\tau})$ is known as $\tau$-parastrophe or $\tau$-parastrophic image of quasigroup $(Q,\ast)$, where $\tau \in S_3$.

%

Let $(Q,\ast)$ be a groupoid and let $c$ be a fixed element of $Q$. The maps $L_c:Q \rightarrow Q$ and $R_c:Q \rightarrow Q$ defined as $L_c x = c\ast x$ and $R_c x = x\ast c$, for all $x\in Q$, are respectively called the \textit{left and the right translations}.

If $(Q,\ast)$ is a quasigroup then it is possible to define third kind of translation, known as \textit{middle translation}, $P_c:Q \rightarrow Q$ defined as $x\ast P_c x=c$, for all $x\in Q$.

In terms of translation maps Definition \ref{qg} can also be written as:\\
A groupoid $(Q,\ast)$ is called a quasigroup if the left and right translations $L_c$ and $R_c$ are bijective maps for all $c\in Q$.

%
%
%
%
%

\begin{definition}
	A groupoid $(G,\circ)$ is an \textit{isotopic image} or \textit{isotope} of a groupoid $(G,\ast)$, if there exist permutations $\alpha_1,\alpha_2,\alpha_3$ of the set $G$ such that 
	\begin{equation} \label{1}
	x \circ y=\alpha_3^{-1}(\alpha_1 x \ast \alpha_2 y)
	\end{equation}
	for all $x,y \in G$. 
\end{definition}

We can also write the equality (\ref{1}) as $(G,\circ)=(G,\ast)R$, where the triplet $R=(\alpha_1,\alpha_2,\alpha_3)$ is called an isotopism or isotopy of the groupoid $(G,\ast)$. It can easily be seen that an isotopic image of a quasigroup is also a quasigroup.

If the two operators $\circ$ and $\ast$ are equal then triplet $R$ is called an \textit{autotopism} or \textit{autotopy} of binary groupoid $(G,\ast)$. Let $Avt(G,\ast)$ denotes the set of all autotopies of a groupoid $(G,\ast)$. It can easily be seen that $Avt(G,\ast)$ forms a group with respect to usual component-wise multiplications of autotopies.


If $\tau \in S_3$ and $R=(\alpha_1,\alpha_2,\alpha_3)$ is an isotopy of a binary groupoid $(Q,\ast)$, then the action of $\tau$ on the triplet $R$ denoted by $R^{\tau}$ shall be defined as $R^{\tau}=(\alpha_{\tau^{-1}1},\alpha_{\tau^{-1}2},\alpha_{\tau^{-1}3})$.  


\begin{definition} \label{isos}
	A quasigroup $(Q,B)$ is an \textit{isostrophic image} or \textit{isostrophe} of $(Q,A)$, if there exists a collection of permutations $(\sigma,(\alpha_1,\alpha_2,\alpha_3))$ =$(\sigma, R)$, where $\sigma \in S_3$ and $R = (\alpha_1,\alpha_2,\alpha_3)$ is triplet of permutations $\alpha_1,\alpha_2,\alpha_3$ of the set $Q$ such that
	\begin{equation} \label{2}
	B(x_1,x_2)=A(x_1,x_2)(\sigma,R)=\alpha_3^{-1}A(\alpha_1 x_{\sigma^{-1}1},\alpha_2 x_{\sigma^{-1}2}) 
	\end{equation}
	for all $x_1,x_2 \in Q$.
\end{definition}

The tuple $(\sigma, R)$ is known as \textit{isostrophism} or \textit{isostrophy} of the quasigroup $(Q,A)$. We can rewrite (\ref{2}) as $B=(A^\sigma)R$. We shall call $\alpha_i$, the $i^\text{th}$ component of the isostrophism $(\sigma,R)$, for $i=1,2,3$.
Definition \ref{isos} can also be written as:\\
An isostrophic image of a quasigroup is defined as an isotopic image of its parastrophe.

Let $(\sigma,R)$ and $(\tau,S)$ be isostrophisms of a quasigroup $(Q,A)$. Then their multiplication is defined as follows:
\begin{equation} \label{aus}
(\sigma,R)(\tau,S)=(\sigma\tau, R^{\tau}S)
\end{equation}
where $A^{\sigma \tau}=(A^\sigma)^\tau$ and $(x_1,x_2,x_3)(R^\tau S)=((x_1,x_2,x_3)R^\tau) S$, for all quasigroup triplets $(x_1,x_2,x_3)$.

It may be noted that this multiplication has been defined in different way by Belousov, Lyakh, Keedwell and Shcherbacov \cite{Belousov&72, Lyakh, Keedwell&Shcherbacov&2004}.




The inverse of an isostrophism $(\sigma,R)$ is:
\begin{equation}
(\sigma,R)^{-1}=(\sigma^{-1},(R^{-1})^{\sigma^{-1}})=(\sigma^{-1},(\alpha^{-1}_{\sigma 1}, \alpha^{-1}_{\sigma 2}, \alpha^{-1}_{\sigma 2}))
\end{equation}

If the binary operations $A$ and $B$ in Definition \ref{isos} are equal then the tuple $(\sigma,R)$ is called an \textit{autostrophism}  or \textit{autostrophy} of the quasigroup $(Q,A)$. Let $Aus(Q,A)$ denotes the set of all autostrophisms of quasigroup $(Q,A)$. $Aus(Q,A)$ forms a group under the multiplication operation defined in (\ref{aus}).


\begin{theorem} \cite{Belousov} \label{thm1}
	If quasigroup $(Q,\circ)$ is an isostrophic image of quasigroup $(Q,\ast)$ with an isostrophy $\theta$, i.e., $(Q,\circ)=(Q,\ast)\theta$, then $Aus(Q,\circ)=\theta^{-1}Aus(Q,\ast)\theta$.
\end{theorem}

\begin{definition}
	Let $(G,\ast)$ be a groupoid and let $c$ be an element of $G$. Then $c$ is left (right, middle) nuclear in $(G,\ast)$ if $L_{c\ast x}=L_x L_c$ ($R_{x\ast c}=R_x R_c, L_{x \ast c}=L_c L_x$) for all $x\in G$.
\end{definition}
If $c$ is left, right and middle nuclear in $(G,\ast)$ then $c$ is called nuclear in groupoid $(G,\ast)$.
\begin{definition}
	The left nucleus $N_l$ (right nucleus $N_r$, middle nucleus $N_m$) of a quasigroup $(Q,\ast)$ is the set of all left (right, middle) nuclear elements in $(Q,\ast)$. Equivalently,
	\[N_l=\{a\in Q\mid a\ast(x\ast y)=(a\ast x)\ast y, \; \forall x, y \in Q\},\]
	\[N_r=\{a\in Q\mid (x\ast y) \ast a=x\ast(y\ast a), \; \forall x, y \in Q\}\] and
	\[N_m=\{a\in Q\mid (x\ast a)\ast y=x\ast(a\ast y), \; \forall x, y \in Q\}.\]
\end{definition}

The nucleus $N$ of the quasigroup $(Q,\ast)$ is defined as $N=N_l\cap N_r \cap N_m.$

If $N_l$ ($N_r$, $N_m$) is non empty, then $N_l$ ($N_r$, $N_m$) is a subgroup of quasigroup $(Q,\ast)$ \cite{Pflugfelder}. It has been shown that there is a weakness in Garrison's nucleus, viz., if a quasigroup $(Q,\ast)$ has a non-trivial (or non-empty) left or right or middle nucleus then the quasigroup is a left loop or a right loop or a loop respectively \cite{Garrison, Pflugfelder}.


\begin{definition} 
	The left (right, middle) A-nucleus of a quasigroup $(Q,\ast)$ is defined as the set of all
	autotopisms of the form $(\alpha,\varepsilon,\gamma)$ ($(\varepsilon,\beta,\gamma)$, $(\alpha,\beta,\varepsilon)$) of the quasigroup $(Q,\ast)$, where $\alpha,\beta,\gamma$ are permutations of the set $Q$ and $\varepsilon$ is the identity mapping.
	
	
\end{definition}

Note that the symbol A in the above definition stands for autotopy. We shall respectively denote these three sets of autotopisms by $N_l^A$, $N_r^A$ and $N_m^A$ and their sets of all $i^\text{th}$ components by $\prescript{}{i}N_l^A$, $\prescript{}{i}N_r^A$ and $\prescript{}{i}N_m^A$ respectively, for $i=1,2,3$.


Table \ref{table2} shows connections between components of the A-nuclei of a quasigroup $(Q,\cdot)$ and components of the A-nuclei of its isostrophic images of the form $(Q,\circ)=(Q,\cdot)(\sigma,R)$, where $\sigma \in S_3$, $R=(\alpha,\beta,\gamma)$ and $\alpha,\beta,\gamma$ are permutations of the set $Q$ \cite{Shcherbacov&2011}. Note that, to denote the A-nuclei of quasigroups $(Q,\circ)$ and $(Q,\cdot)$ the symbol $A$ (for autotopy) is replaced by the binary operations $\circ$ and $\cdot$ respectively.

\begin{table} [htp]
	\caption{Connections between components of A-nuclei of a quasigroup and its isostrophic images.} \label{table2}
	\centering
	\begin{adjustbox}{width=\textwidth,center}
		\begin{tabular}{|c|c|c|c|c|c|c|} \hline
			&$(\varepsilon,R)$ &$((1\,2),R)$ &$((1\,3),R)$ &$((2\,3),R)$&$((1\,3\,2),R)$ &$((1\,2\,3),R)$ \\ \hline\hline
			
			$\prescript{}{1}N_l^\circ$& $\alpha^{-1} \prescript{}{1}N^{\cdot}_{l} \alpha$& $\alpha^{-1}\prescript{}{2}N_r^\cdot \alpha$&$\alpha^{-1}\prescript{}{3}N_l^\cdot \alpha$&$\alpha^{-1}\prescript{}{1}N_m^\cdot \alpha$&$\alpha^{-1}\prescript{}{2}N_m^\cdot \alpha$&$\alpha^{-1}\prescript{}{3}N_r^\cdot \alpha$ \\ \hline
			
			$\prescript{}{3}N_l^\circ$& $\gamma^{-1} \prescript{}{3}N^{\cdot}_l \gamma$& $\gamma^{-1}\prescript{}{3}N_r^\cdot \gamma$&$\gamma^{-1}\prescript{}{1}N_l^\cdot \gamma$&$\gamma^{-1}\prescript{}{2}N_m^\cdot \gamma$&$\gamma^{-1}\prescript{}{1}N_m^\cdot \gamma$&$\gamma^{-1}\prescript{}{2}N_r^\cdot \gamma$ \\ \hline
			
			$\prescript{}{2}N_r^\circ$& $\beta^{-1} \prescript{}{2}N^{\cdot}_r \beta$& $\beta^{-1}\prescript{}{1}N_l^\cdot \beta$&$\beta^{-1}\prescript{}{2}N_m^\cdot \beta$&$\beta^{-1}\prescript{}{3}N_r^\cdot \beta$&$\beta^{-1}\prescript{}{3}N_l^\cdot \beta$&$\beta^{-1}\prescript{}{1}N_m^\cdot \beta$ \\ \hline
			
			$\prescript{}{3}N_r^\circ$& $\gamma^{-1}\prescript{}{3}N^{\cdot}_r \gamma$& $\gamma^{-1}\prescript{}{3}N_l^\cdot \gamma$&$\gamma^{-1}\prescript{}{1}N_m^\cdot \gamma$&$\gamma^{-1}\prescript{}{2}N_r^\cdot \gamma$&$\gamma^{-1}\prescript{}{1}N_l^\cdot \gamma$&$\gamma^{-1} \prescript{}{2}N_m^\cdot \gamma$ \\ \hline
			
			$\prescript{}{1}N_m^\circ$& $\alpha^{-1} \prescript{}{1}N^{\cdot}_m \alpha$& $\alpha^{-1}\prescript{}{2}N_m^\cdot \alpha$&$\alpha^{-1}\prescript{}{3}N_r^\cdot \alpha$&$\alpha^{-1}\prescript{}{1}N_l^\cdot \alpha$&$\alpha^{-1}\prescript{}{2}N_r^\cdot \alpha$&$\alpha^{-1}\prescript{}{3}N_l^\cdot \alpha$ \\ \hline
			
			$\prescript{}{2}N_m^\circ$& $\beta^{-1} \prescript{}{2}N^{\cdot}_m \beta$& $\beta^{-1}\prescript{}{1}N_m^\cdot \beta$&$\beta^{-1}\prescript{}{2}N_r^\cdot \beta$&$\beta^{-1}\prescript{}{3}N_l^\cdot \beta$&$\beta^{-1}\prescript{}{3}N_r^\cdot \beta$&$\beta^{-1}\prescript{}{1}N_l^\cdot \beta$ \\ \hline
			
		\end{tabular}
	\end{adjustbox}
\end{table} 

The left, right and middle A-nuclei of a quasigroup have been generalized to the left, right and middle $\sigma$-A-nuclei respectively, for $\sigma\in S_3$, as follows:

\begin{definition}  \cite{Chauhan22} \label{sigma}
	The left (right, middle) $\sigma$-A-nuclei of a quasigroup $(Q,\cdot)$ is defined as the set of all autostrophisms of the form $(\sigma,(\alpha,\varepsilon,\gamma))$ ($(\sigma,(\varepsilon,\beta,\gamma))$, $(\sigma,(\alpha,\beta,\varepsilon))$) of the quasigroup $(Q,\cdot)$, where $\alpha,\beta,\gamma$ are permutations of the set $Q$, $\varepsilon$ is the identity mapping and $\sigma \in S_3$. 
\end{definition}

Note that, the left, right and middle $\varepsilon$-A-nuclei of a quasigroup can be considered as the left, right and middle A-nuclei of the quasigroup respectively. 

We shall denote the left, right and middle $\sigma$-A-nuclei by $^{\sigma}N_l^A$, $^{\sigma}N_r^A$ and $^{\sigma}N_m^A$, respectively.
Let $\prescript{\sigma}{i}{N}_l^A$, $\prescript{\sigma}{i}{N}_r^A$ and $\prescript{\sigma}{i}{N}_m^A$ denote the respective sets of all $i^\text{th}$ components of members of the left, right and middle $\sigma$-A-nuclei, for $i=1,2,3$. 

It may be noted that $\prescript{\sigma}{2}{N}_l^A$ = $\prescript{\sigma}{1}{N}_r^A = \prescript{\sigma}{3}{N}_m^A = \{\varepsilon\}$. We shall call these sets as trivial component sets.

\sloppy
\begin{remark} \label{r1}	
It may be noted that $(\sigma,(\alpha_1,\alpha_2,\alpha_3))\in \prescript{\sigma}{}{N}_l^A$ iff $(\sigma^{-1},(\alpha_{\sigma^{-1}1},\alpha_{\sigma^{-1}2},\alpha_{\sigma^{-1}3}))\in \prescript{\sigma^{-1}}{}{N}_v^A$, where $v=
	\begin{cases*}
	r & if $\sigma^{-1} 2=1$,\\ 
	l & if $\sigma^{-1} 2=2$, \\
	m & if $\sigma^{-1} 2=3.$
	\end{cases*}$ Thus we have
	\begin{equation}
	\prescript{\sigma}{}{N}_l^A \neq \emptyset \text{ iff } \prescript{\sigma^{-1}}{}{N}_v^A \neq \emptyset, \text{ for all } \sigma\in S_3. 
	\end{equation}	
\end{remark}

\section{Properties of $\sigma$-A-nuclei} \label{sec3}
In this section we characterize the inverse sets of $\sigma$-A-nuclei and the products of two $\sigma$-A-nuclei of a quasigroup, and their respective component sets.
Further, we find connections between $\sigma$-A-nuclei of a quasigroup and $\sigma$-A-nuclei of its isostrophic images and parastrophic images, and their respective component sets.  


Recall that if $G$ is a group and $B\subseteq G$ is any subset, then we denote by $B^{-1}$ the set of inverses of elements of $B$, i.e. $B^{-1}=\{x^{-1} \mid x\in B\}$.
Let $(Q,\cdot)$ be a quasigroup and $G=Aus(Q,\cdot)$, the group all autostrophies of $(Q,\cdot)$. 
For each permutation $\sigma\in S_3$, we shall characterize $B^{-1}$, where $B=\prescript{\sigma}{}N_v^A \subseteq Aus(Q,\cdot)$, for $v\in \{l,m,r\}$.
For more clarity in notations and result, let us consider a particular example.

Let $\sigma=(1\,2)$ and $v=r$.
We claim that 
\begin{equation} \label{6}
\left(\prescript{(1\,2)}{}N_r^A\right)^{-1}=\prescript{(1\,2)}{}N_l^A.
\end{equation}
In view of Remark \ref{r1}, the result is trivial for the case  $\prescript{(1\,2)}{}N_r^A=\emptyset$. 
So, we prove (\ref{6}) for $\prescript{(1\,2)}{}N_r^A\neq \emptyset$.
We shall show the inclusions in both directions.
Let $\varphi \in \left(\prescript{(1\,2)}{}N_r^A\right)^{-1}$.
Then we have $\varphi^{-1} \in \prescript{(1\,2)}{}N_r^A$. 
Thus there exist permutations $\alpha_2,\alpha_3$ of $Q$ such that
$\varphi^{-1}=((1\,2),(\varepsilon,\alpha_2,\alpha_3))$, which implies 
\begin{equation} \label{7}
\varphi=((1\,2),(\alpha_2^{-1},\varepsilon, \alpha_3^{-1})) \in \prescript{(1\,2)}{}N_l^A.
\end{equation}
Thus $\left(\prescript{(1\,2)}{}N_r^A\right)^{-1} \subseteq \prescript{(1\,2)}{}N_l^A$.
The other way inclusion can be shown on similar lines.

Also, from (\ref{7}) we have $\alpha_2^{-1} \in \prescript{(1\,2)}{1}N_l^A $ and $\alpha_3^{-1} \in \prescript{(1\,2)}{3}N_l^A $, which implies  $\left(\prescript{(1\,2)}{2}N_r^A\right)^{-1} \subseteq \prescript{(1\,2)}{1}N_l^A$ and $\left(\prescript{(1\,2)}{3}N_r^A\right)^{-1} \subseteq \prescript{(1\,2)}{3}N_l^A$.
Similarly, $\prescript{(1\,2)}{1}N_l^A  \subseteq  \left(\prescript{(1\,2)}{2}N_r^A\right)^{-1}$ and $\prescript{(1\,2)}{3}N_l^A \subseteq  \left(\prescript{(1\,2)}{3}N_r^A\right)^{-1} $. Thus 
\[\left(\prescript{(1\,2)}{2}N_r^A\right)^{-1}=\prescript{(1\,2)}{1}N_l^A \text{ and } \left(\prescript{(1\,2)}{3}N_r^A\right)^{-1}= \prescript{(1\,2)}{3}N_l^A.\]

Note that we shall omit the relations for the trivial component sets.

The following theorem gives a complete description of the inverse sets of $\sigma$-A-nuclei of a quasigroup and their respective component sets, for any $\sigma \in S_3$.

\begin{theorem} \label{thm2}
	If $(Q,\cdot)$ is a quasigroup, then for all $\sigma\in S_3$:
	\begin{enumerate}
		\item $\left(\prescript{\sigma}{}N_l^A\right)^{-1}=\prescript{\sigma^{-1}}{}N_v^A$,	where $v=
		\begin{cases*}
		r & if $\sigma^{-1} 2=1$,\\ 
		l & if $\sigma^{-1} 2=2$, \\
		m & if $\sigma^{-1} 2=3.$
		\end{cases*}$ 
		Further, $\left(\prescript{\sigma}{1}N^A_l\right)^{-1}=\prescript{\sigma^{-1}}{\sigma^{-1}1}N_v^A$ and $\left(\prescript{\sigma}{3}N^A_l\right)^{-1}=\prescript{\sigma^{-1}}{\sigma^{-1}3}N_v^A$.
		
		\item $\left(\prescript{\sigma}{}N_r^A\right)^{-1}=\prescript{\sigma^{-1}}{}N_v^A$,	where $v=
		\begin{cases*}
		r & if $\sigma^{-1} 1=1$,\\ 
		l & if $\sigma^{-1} 1=2$, \\
		m & if $\sigma^{-1} 1=3.$
		\end{cases*}$
		Further, $\left(\prescript{\sigma}{2}N^A_r\right)^{-1}=\prescript{\sigma^{-1}}{\sigma^{-1}2}N_v^A$ and $\left(\prescript{\sigma}{3}N^A_r\right)^{-1}=\prescript{\sigma^{-1}}{\sigma^{-1}3}N_v^A$.
		\item $\left(\prescript{\sigma}{}N_m^A\right)^{-1}=\prescript{\sigma^{-1}}{}N_v^A$,	where $v=
		\begin{cases*}
		r & if $\sigma^{-1} 3=1$,\\ 
		l & if $\sigma^{-1} 3=2$, \\
		m & if $\sigma^{-1} 3=3.$
		\end{cases*}$
		Further, $\left(\prescript{\sigma}{1}N^A_m\right)^{-1}=\prescript{\sigma^{-1}}{\sigma^{-1}1}N_v^A$ and $\left(\prescript{\sigma}{2}N^A_m\right)^{-1}=\prescript{\sigma^{-1}}{\sigma^{-1}2}N_v^A$.
		
	\end{enumerate}
\end{theorem}

\begin{proof}
	$(1)$ From Remark \ref{r1} it is clear that $\prescript{\sigma}{}N_l^A= \emptyset$ iff $\prescript{\sigma^{-1}}{}N_v^A= \emptyset $, so we are done.
	We prove the result for $\prescript{\sigma}{}N_l^A \neq \emptyset $.
	We shall show the inclusions in both directions.
	Let $\varphi \in \left(\prescript{\sigma}{}N_l^A\right)^{-1}$. 
	Then $\varphi^{-1} \in \prescript{\sigma}{}N_l^A$. 
	Thus there exist permutations $\alpha_1,\alpha_2,\alpha_3$ of $Q$ such that $\varphi^{-1}=(\sigma,(\alpha_1,\alpha_2,\alpha_3))$, where $\alpha_2=\varepsilon$.
	Hence 
	\begin{equation}
	\varphi=(\varphi^{-1})^{-1}=(\sigma^{-1},(\alpha^{-1}_{\sigma 1}, \alpha^{-1}_{\sigma 2}, \alpha^{-1}_{\sigma 3})).
	\end{equation} 
	Let $\varphi=(\sigma^{-1},(\varphi_1,\varphi_2,\varphi_3))$.
	On comparing the components we get $\varphi_i=\alpha^{-1}_{\sigma i}$, i.e.
	\begin{equation}\label{9}
	\varphi_{\sigma^{-1}i}=\alpha^{-1}_i, \text{ for } i=1,2,3.
	\end{equation} 
	When $i=2$, $\varphi_{\sigma^{-1}2}=\alpha^{-1}_2=\varepsilon$, which implies that $(\sigma^{-1}2)^\text{{th}}$ component of $\varphi$ is identity. 
	This shows $\varphi \in \prescript{\sigma^{-1}}{}N_v^A$ and hence
	$\left(\prescript{\sigma}{}N_l^A\right)^{-1} \subseteq \prescript{\sigma^{-1}}{}N_v^A.$
	Also, from (\ref{9}) and $\varphi=(\sigma^{-1},(\varphi_1,\varphi_2,\varphi_3)) \in \prescript{\sigma^{-1}}{}N_v^A$ we get $\alpha_i^{-1} \in \prescript{\sigma^{-1}}{\sigma^{-1}i}N_v^A $, which implies $\left(\prescript{\sigma}{i}N^A_l\right)^{-1} \subseteq \prescript{\sigma^{-1}}{\sigma^{-1}i}N_v^A$, for $i=1,3.$

	
	Conversely, let $\varphi \in \prescript{\sigma^{-1}}{}N_v^A$. 
	Then there exist permutations $\varphi_1,\varphi_2,\varphi_3$ of $Q$ such that $\varphi=(\sigma^{-1},(\varphi_1,\varphi_2,\varphi_3))$, where $\varphi_{\sigma^{-1}2}=\varepsilon$.
	Thus
	\begin{equation}\label{10}
	\varphi^{-1}=(\sigma,(\varphi^{-1}_{\sigma^{-1}1},\varphi^{-1}_{\sigma^{-1}2}, \varphi^{-1}_{\sigma^{-1}3}))=(\sigma,(\varphi^{-1}_{\sigma^{-1}1},\varepsilon,\varphi^{-1}_{\sigma^{-1}3})) \in \prescript{\sigma}{}N_l^A.
	\end{equation} 
	This shows $\varphi \in \left(\prescript{\sigma}{}N_l^A\right)^{-1}$ and hence $ \prescript{\sigma^{-1}}{}N_v^A \subseteq \left(\prescript{\sigma}{}N_l^A\right)^{-1}.$ 
	Also, from (\ref{10}) we get $\varphi_{\sigma^{-1}i} \in \left(\prescript{\sigma}{i}N^A_l\right)^{-1}$, which gives $\prescript{\sigma^{-1}}{\sigma^{-1}1}N_v^A \subseteq \left(\prescript{\sigma}{1}N^A_l\right)^{-1}$, for $i=1,3$ (since $\varphi_i \in \prescript{\sigma^{-1}}{i}N_v^A$). 
	
	$(2)$ and $(3)$ can be proved on similar lines.
\end{proof}

Using Theorem \ref{thm2} the components of inverse sets of a left, right and middle $\sigma$-A-nuclei can be found from the components of left, right and middle $\sigma^{-1}$-A-nuclei of a quasigroup and vice versa.
Table \ref{table3} shows relationships between components of $\sigma$-A-nuclei and $\sigma^{-1}$-A-nuclei of a quasigroup $(Q,\cdot)$ obtained from Theorem \ref{thm2}, for all $\sigma \in S_3$.

\begin{table}[h!]
	\caption{Relationships between components of $\sigma$-A-nuclei and $\sigma^{-1}$-A-nuclei.} \label{table3}
	\centering
	\begin{adjustbox}{width=\textwidth,center}
		\begin{tabular}{|l|l|l|l|} \hline
			$\sigma$ & & & \\ \hline\hline
			\multirow{2}{*}{$\varepsilon$} & \multicolumn{1}{|l}{$\left(\prescript{}{1}N^A_l\right)^{-1}=\prescript{}{1}N_l^A$} & \multicolumn{1}{|l}{$\left(\prescript{}{2}N^A_r\right)^{-1}=\prescript{}{2}N_r^A$} & \multicolumn{1}{|l|}{$\left(\prescript{}{1}N^A_m\right)^{-1}=\prescript{}{1}N_m^A$} \\\cline{2-4}
			& \multicolumn{1}{|l}{$\left(\prescript{}{3}N^A_l\right)^{-1}=\prescript{}{3}N_l^A$} & \multicolumn{1}{|l}{$\left(\prescript{}{3}N^A_r\right)^{-1}=\prescript{}{3}N_r^A$} & \multicolumn{1}{|l|}{$\left(\prescript{}{2}N^A_m\right)^{-1}=\prescript{}{2}N_m^A$}  \\ \hline
			
			\multirow{2}{*}{$(1\,2)$} &  \multicolumn{1}{|l}{$\left(\prescript{(1\,2)}{1}N^A_l\right)^{-1}=\prescript{(1\,2)}{2}N_r^A$} &
			\multicolumn{1}{|l}{$\left(\prescript{(1\,2)}{2}N^A_r\right)^{-1}=\prescript{(1\,2)}{1}N_l^A$} & \multicolumn{1}{|l|}{$\left(\prescript{(1\,2)}{2}N^A_m\right)^{-1}=\prescript{(1\,2)}{1}N_m^A$} \\\cline{2-4}
			&  \multicolumn{1}{|l}{$\left(\prescript{(1\,2)}{3}N^A_l\right)^{-1}=\prescript{(1\,2)}{3}N_r^A$} &
			\multicolumn{1}{|l}{$\left(\prescript{(1\,2)}{3}N^A_r\right)^{-1}=\prescript{(1\,2)}{3}N_l^A$} & \multicolumn{1}{|l|}{$\left(\prescript{(1\,2)}{1}N^A_m\right)^{-1}=\prescript{(1\,2)}{2}N_m^A$}  \\ \hline
			
			\multirow{2}{*}{$(1\,3)$} & \multicolumn{1}{|l}{$\left(\prescript{(1\,3)}{3}N^A_l\right)^{-1}\prescript{(1\,3)}{1}N_l^A$}  & \multicolumn{1}{|l}{$\left(\prescript{(1\,3)}{3}N^A_r\right)^{-1}=\prescript{(1\,3)}{1}N_m^A$} & \multicolumn{1}{|l|}{$\left(\prescript{(1\,3)}{2}N^A_m\right)^{-1}=\prescript{(1\,3)}{2}N_r^A$} \\\cline{2-4}
			& \multicolumn{1}{|l}{$\left(\prescript{(1\,3)}{1}N^A_l\right)^{-1}=\prescript{(1\,3)}{3}N_l^A$} & \multicolumn{1}{|l}{$\left(\prescript{(1\,3)}{2}N^A_r\right)^{-1}=\prescript{(1\,3)}{2}N_m^A$} 
			& \multicolumn{1}{|l|}{$\left(\prescript{(1\,3)}{1}N^A_m\right)^{-1}=\prescript{(1\,3)}{3}N_r^A$}  \\ \hline
			
			\multirow{2}{*}{$(2\,3)$} &
			\multicolumn{1}{|l}{$\left(\prescript{(2\,3)}{1}N^A_l\right)^{-1}=\prescript{(2\,3)}{1}N_m^A$}  & \multicolumn{1}{|l}{$\left(\prescript{(2\,3)}{3}N^A_r\right)^{-1}=\prescript{(2\,3)}{2}N_r^A$} & 
			\multicolumn{1}{|l|}{$\left(\prescript{(2\,3)}{1}N^A_m\right)^{-1}=\prescript{(2\,3)}{1}N_l^A$} \\\cline{2-4}
			&
			\multicolumn{1}{|l}{$\left(\prescript{(2\,3)}{3}N^A_l\right)^{-1}=\prescript{(2\,3)}{2}N_m^A$} &
			\multicolumn{1}{|l}{$\left(\prescript{(2\,3)}{2}N^A_r\right)^{-1}=\prescript{(2\,3)}{3}N_r^A$} & 
			\multicolumn{1}{|l|}{$\left(\prescript{(2\,3)}{2}N^A_m\right)^{-1}=\prescript{(2\,3)}{3}N_l^A$} \\ \hline
			
			\multirow{2}{*}{$(1\,2\,3)$} & \multicolumn{1}{|l}{$\left(\prescript{(1\,2\,3)}{3}N^A_l\right)^{-1}=\prescript{(1\,3\,2)}{2}N_r^A$} & \multicolumn{1}{|l}{$\left(\prescript{(1\,2\,3)}{2}N^A_r\right)^{-1}=\prescript{(1\,3\,2)}{1}N_m^A$} & \multicolumn{1}{|l|}{$\left(\prescript{(1\,2\,3)}{2}N^A_m\right)^{-1}=\prescript{(1\,3\,2)}{1}N_l^A$}  \\\cline{2-4}
			& \multicolumn{1}{|l}{$\left(\prescript{(1\,2\,3)}{1}N^A_l\right)^{-1}=\prescript{(1\,3\,2)}{3}N_r^A$} & \multicolumn{1}{|l}{$\left(\prescript{(1\,2\,3)}{3}N^A_r\right)^{-1}=\prescript{(1\,3\,2)}{2}N_m^A$} 	& \multicolumn{1}{|l|}{$\left(\prescript{(1\,2\,3)}{1}N^A_m\right)^{-1}=\prescript{(1\,3\,2)}{3}N_l^A$} \\ \hline
			
			\multirow{2}{*}{$(1\,3\,2)$} & \multicolumn{1}{|l}{$\left(\prescript{(1\,3\,2)}{3}N^A_l\right)^{-1}=\prescript{(1\,2\,3)}{1}N_m^A$} & \multicolumn{1}{|l}{$\left(\prescript{(1\,3\,2)}{3}N^A_r\right)^{-1}=\prescript{(1\,2\,3)}{1}N_l^A$} & \multicolumn{1}{|l|}{$\left(\prescript{(1\,3\,2)}{1}N^A_m\right)^{-1}=\prescript{(1\,2\,3)}{2}N_r^A$} \\\cline{2-4}
			& \multicolumn{1}{|l}{$\left(\prescript{(1\,3\,2)}{1}N^A_l\right)^{-1}=\prescript{(1\,2\,3)}{2}N_m^A$}& \multicolumn{1}{|l}{$\left(\prescript{(1\,3\,2)}{2}N^A_r\right)^{-1}=\prescript{(1\,2\,3)}{3}N_l^A$} & \multicolumn{1}{|l|}{$\left(\prescript{(1\,3\,2)}{2}N^A_m\right)^{-1}=\prescript{(1\,2\,3)}{3}N_r^A$}  \\ \hline
			
		\end{tabular}
	\end{adjustbox}
\end{table}

Recall that if $G$ is a group and $B$ and $C$ are subsets of $G$, then the product of $B$ and $C$ is the subset of $G$ defined as:
\[BC=\{bc \mid b \in B \text{ and } c\in C\}.\]
It may be noted that $BC=\emptyset$ iff at least one of $B$, $C=\emptyset$.

Let $(Q,\cdot)$ be a quasigroup and $G=Aus(Q,\cdot)$. 
For each permutation $\sigma, \tau \in S_3$, we shall characterize $BC$, where $B=\prescript{\sigma}{}N_u^A \subseteq Aus(Q,\cdot)$ and $C=\prescript{\tau}{}N_v^A \subseteq Aus(Q,\cdot)$, where $u,v\in \{l,m,r\}$.

The following theorem gives a complete description of products of $\sigma$-A-nuclei and $\tau$-A-nuclei of a quasigroup and their respective component sets, for any $\sigma,\tau \in S_3$.

\begin{theorem} \label{thm3}
	If $(Q,\cdot)$ is a quasigroup, then for all $\sigma, \tau \in S_3$:
	\begin{enumerate}
		\item $\prescript{\sigma}{}N_v^A \prescript{\tau}{}N_l^A= \prescript{\sigma \tau}{}N_l^A$, where $v=
		\begin{cases*}
		r & if $\tau^{-1} 2=1$,\\ 
		l & if $\tau^{-1} 2=2$, \\
		m & if $\tau^{-1} 2=3$,
		\end{cases*}$ provided  $\prescript{\sigma}{}N_v^A, \prescript{\tau}{}N_l^A \neq \emptyset$. 
		Further, 
		$\prescript{\sigma}{\tau^{-1}1}N_v^A \prescript{\tau}{1}N_l^A= \prescript{\sigma \tau}{1}N_l^A$ and $\prescript{\sigma}{\tau^{-1}3}N_v^A \prescript{\tau}{3}N_l^A= \prescript{\sigma \tau}{3}N_l^A$. 
		
		\item $\prescript{\sigma}{}N_v^A \prescript{\tau}{}N_r^A= \prescript{\sigma \tau}{}N_r^A$, where $v=
		\begin{cases*}
		r & if $\tau^{-1} 1=1$,\\ 
		l & if $\tau^{-1} 1=2$, \\
		m & if $\tau^{-1} 1=3$,
		\end{cases*}$ provided  $\prescript{\sigma}{}N_v^A, \prescript{\tau}{}N_r^A \neq \emptyset$.
		Further, 
		$\prescript{\sigma}{\tau^{-1}2}N_v^A \prescript{\tau}{2}N_r^A= \prescript{\sigma \tau}{2}N_r^A$ and $\prescript{\sigma}{\tau^{-1}3}N_v^A \prescript{\tau}{3}N_r^A= \prescript{\sigma \tau}{3}N_r^A$. 
		
		\item $\prescript{\sigma}{}N_v^A \prescript{\tau}{}N_m^A= \prescript{\sigma \tau}{}N_m^A$, where $v=
		\begin{cases*}
		r & if $\tau^{-1} 3=1$,\\ 
		l & if $\tau^{-1} 3=2$, \\
		m & if $\tau^{-1} 3=3$,
		\end{cases*}$ provided  $\prescript{\sigma}{}N_v^A, \prescript{\tau}{}N_m^A \neq \emptyset$. Further, 
		$\prescript{\sigma}{\tau^{-1}1}N_v^A \prescript{\tau}{1}N_m^A= \prescript{\sigma \tau}{1}N_m^A$ and $\prescript{\sigma}{\tau^{-1}2}N_v^A \prescript{\tau}{2}N_m^A= \prescript{\sigma \tau}{2}N_m^A$. 
	\end{enumerate}
\end{theorem}

\begin{proof}
	$(1)$ 
	We shall show the inclusions in both directions. 	  
	Let $\phi \in \prescript{\sigma}{}N_v^A$ and $\psi \in \prescript{\tau}{}N_l^A$. 
	Then there exist permutations $\phi_1,\phi_2,\phi_3$ and $\psi_1,\psi_3$ of $Q$ such that $\phi=(\sigma,(\phi_1,\phi_2,\phi_3))$ and $\psi=(\tau,(\psi_1,\varepsilon,\psi_3))$, where $\phi_{\tau^{-1}2}=\varepsilon$.
	Then 
	\begin{align} 
	\phi \psi & = (\sigma \tau,(\phi_{\tau^{-1}1} \psi_1, \phi_{\tau^{-1}2}, \phi_{\tau^{-1}3} \psi_3))\\
	&= (\sigma \tau,(\phi_{\tau^{-1}1} \psi_1, \varepsilon, \phi_{\tau^{-1}3} \psi_3)) \in \prescript{\sigma \tau}{}N_l^A \label{12}
	\end{align}
	(because $\phi \psi \in Aus(Q,\cdot)$).
	Thus $\prescript{\sigma}{}N_v^A \prescript{\tau}{}N_l^A  \subseteq \prescript{\sigma \tau}{}N_l^A$. This gives $\prescript{\sigma \tau}{}N_l^A \neq  \emptyset$, since $\prescript{\sigma}{}N_v^A, \prescript{\tau}{}N_l^A \neq \emptyset$. 
	Also from (\ref{12}) we obtain $\phi_{\tau^{-1}1} \psi_1 \in \prescript{\sigma \tau}{1}N_l^A$  and $\phi_{\tau^{-1}3} \psi_3 \in \prescript{\sigma \tau}{3}N_l^A$, which implies $\prescript{\sigma}{\tau^{-1}1}N_v^A \prescript{\tau}{1}N_l^A \subseteq \prescript{\sigma \tau}{1}N_l^A$  and $\prescript{\sigma}{\tau^{-1}3}N_v^A \prescript{\tau}{3}N_l^A \subseteq \prescript{\sigma \tau}{3}N_l^A$.
	
	Conversely, suppose that $\varphi= (\sigma \tau,S) \in \prescript{\sigma \tau}{}N_l^A$, where $S= (\alpha_1,\varepsilon, \alpha_3)$ for some permutations $\alpha_1,\alpha_3$ of $Q$. 
	Since $\prescript{\sigma}{}N_v^A \neq \emptyset$, there exists an element say $\phi \in \prescript{\sigma}{}N_v^A$. 
	Then there exist permutations $\phi_1,\phi_2,\phi_3$ such that $\phi= (\sigma,(\phi_1,\phi_2,\phi_3))$, where $\phi_{\tau^{-1}2}=\varepsilon$.
	Let $T=(\phi_1,\phi_2,\phi_3)$. 
	As we know that  $\phi, \varphi \in Aus(Q,\cdot)$ and $Aus(Q,\cdot)$ is a group, we have $\phi^{-1} \varphi \in Aus(Q,\cdot)$. 
	Also \[\phi^{-1} \varphi= (\sigma, T)^{-1} (\sigma \tau,S)= (\sigma^{-1}, (T^{-1})^{\sigma^{-1}}) (\sigma \tau,S)= (\tau, (T^{-1})^{\tau} S).\]
	On substituting the values of $T$ and $S$, we get
	\begin{align*}
	\phi^{-1} \rho&= (\tau,(\phi_{\tau^{-1}1}^{-1}\alpha_1,\phi_{\tau^{-1} 2}^{-1}, \phi_{\tau^{-1}3}^{-1} \alpha_3))\\&=  (\tau,((\phi_{\tau^{-1}1}^{-1}\alpha_1,\varepsilon, \phi_{\tau^{-1}3}^{-1} \alpha_3))= \psi \text{(say)}  \in \prescript{\tau}{}N_l^A
	\end{align*}
	(because the second component is identity).
	Hence, for any element $\varphi \in \prescript{\sigma \tau}{}N_l^A$, there exists an element $\phi \in \prescript{\sigma}{}N_v^A$ and an element $\psi \in \prescript{\tau}{}N_l^A$ such that $\phi \psi= \varphi$, i.e., $ \prescript{\sigma \tau}{}N_l^A \subseteq \prescript{\sigma}{}N_v^A \prescript{\tau}{}N_l^A$. 
	Also we get for $\alpha_i\in \prescript{\sigma \tau}{i}N_l^A$ there exists $\phi_{\tau^{-1}i} \in \prescript{\sigma}{\tau^{-1}1}N_v^A$ and $\phi_{\tau^{-1}i}^{-1}\alpha_i\in \prescript{\tau}{i}N_l^A  $ such that $(\phi_{\tau^{-1}i}) (\phi_{\tau^{-1}i}^{-1} \alpha_i)=\alpha_i $, for $i=1,3$. 
	Thus  
	$ \prescript{\sigma \tau}{1}N_l^A \subseteq \prescript{\sigma}{\tau^{-1}1}N_v^A \prescript{\tau}{1}N_l^A $ and  $ \prescript{\sigma \tau}{3}N_l^A \subseteq \prescript{\sigma}{\tau^{-1}3}N_v^A \prescript{\tau}{3}N_l^A $.
	
	$(2)$ and $(3)$ can be proved on similar lines.
\end{proof}

Table \ref{table4} is obtained using Theorem \ref{thm3} that shows relationships between components of $\sigma$-A-nuclei, $\tau$-A-nuclei and $\sigma \tau$-A-nuclei of a quasigroup $(Q,\cdot)$, for all $\sigma, \tau \in S_3$. 
\afterpage{%
	\clearpage
	\thispagestyle{empty}
	\begin{landscape}
		\begin{table}
			\caption{Relationships between components of $\sigma$-A-nuclei, $\tau$-A-nuclei and $\sigma \tau$-A-nuclei.} \label{table4}
			\centering
			\scalebox{0.7}{
				\begin{tabular}{|ll|l|l|l|l|l|} \hline
					$\sigma$$\backslash \tau$ &\multicolumn{1}{|c}{$\varepsilon$} & \multicolumn{1}{|c}{$(1\,2)$} & \multicolumn{1}{|c}{$(1\,3)$}  & \multicolumn{1}{|c}{$(2\,3)$}  & \multicolumn{1}{|c}{$(1\,2\,3)$}  & \multicolumn{1}{|c|}{$(1\,3\,2)$} \\\hline 	\hline
					
					\multirow{6}{*}{$\varepsilon$}  & \multicolumn{1}{|l}{$\prescript{\varepsilon}{1}N_l^A \prescript{\varepsilon}{1}N_l^A= \prescript{\varepsilon}{1}N_l^A$} & \multicolumn{1}{|l}{$\prescript{\varepsilon}{2}N_r^A \prescript{(1\,2)}{1}N_l^A= \prescript{(1\,2)}{1}N_l^A$} & \multicolumn{1}{|l}{$\prescript{\varepsilon}{3}N_l^A \prescript{(1\,3)}{1}N_l^A= \prescript{(1\,3)}{1}N_l^A$} & \multicolumn{1}{|l}{$\prescript{\varepsilon}{1}N_m^A \prescript{(2\,3)}{1}N_l^A= \prescript{(2\,3)}{1}N_l^A$} & \multicolumn{1}{|l}{$\prescript{\varepsilon}{3}N_r^A \prescript{(1\,2\,3)}{1}N_l^A= \prescript{(1\,2\,3)}{1}N_l^A$} & \multicolumn{1}{|l|}{$\prescript{\varepsilon}{2}N_m^A \prescript{(1\,3\,2)}{1}N_l^A= \prescript{(1\,3\,2)}{1}N_l^A$} \\\cline{2-7}
					
					& \multicolumn{1}{|l}{$\prescript{\varepsilon}{3}N_l^A \prescript{\varepsilon}{3}N_l^A= \prescript{\varepsilon}{3}N_l^A$} & \multicolumn{1}{|l}{$\prescript{\varepsilon}{3}N_r^A \prescript{(1\,2)}{3}N_l^A= \prescript{(1\,2)}{3}N_l^A$} & \multicolumn{1}{|l}{$\prescript{\varepsilon}{1}N_l^A \prescript{(1\,3)}{3}N_l^A= \prescript{(1\,3)}{3}N_l^A$} & \multicolumn{1}{|l}{$\prescript{\varepsilon}{2}N_m^A \prescript{(2\,3)}{3}N_l^A= \prescript{(2\,3)}{3}N_l^A$} & \multicolumn{1}{|l}{$\prescript{\varepsilon}{2}N_r^A \prescript{(1\,2\,3)}{3}N_l^A= \prescript{(1\,2\,3)}{3}N_l^A$} & \multicolumn{1}{|l|}{$\prescript{\varepsilon}{1}N_m^A \prescript{(1\,3\,2)}{3}N_l^A= \prescript{(1\,3\,2)}{3}N_l^A$} \\\cline{2-7}
					
					& \multicolumn{1}{|l}{$\prescript{\varepsilon}{2}N_r^A \prescript{\varepsilon}{2}N_r^A= \prescript{\varepsilon}{2}N_r^A$} & \multicolumn{1}{|l}{$\prescript{\varepsilon}{1}N_l^A \prescript{(1\,2)}{2}N_r^A= \prescript{(1\,2)}{2}N_r^A$} & \multicolumn{1}{|l}{$\prescript{\varepsilon}{2}N_m^A \prescript{(1\,3)}{2}N_r^A= \prescript{(1\,3)}{2}N_r^A$} & \multicolumn{1}{|l}{$\prescript{\varepsilon}{3}N_r^A \prescript{(2\,3)}{2}N_r^A= \prescript{(2\,3)}{2}N_r^A$} & \multicolumn{1}{|l}{$\prescript{\varepsilon}{1}N_m^A \prescript{(1\,2\,3)}{2}N_r^A= \prescript{(1\,2\,3)}{2}N_r^A$} & \multicolumn{1}{|l|}{$\prescript{\varepsilon}{3}N_l^A \prescript{(1\,3\,2)}{2}N_r^A= \prescript{(1\,3\,2)}{2}N_r^A$} \\\cline{2-7}
					
					& \multicolumn{1}{|l}{$\prescript{\varepsilon}{3}N_r^A \prescript{\varepsilon}{3}N_r^A= \prescript{\varepsilon}{3}N_r^A$} & \multicolumn{1}{|l}{$\prescript{\varepsilon}{3}N_l^A \prescript{(1\,2)}{3}N_r^A= \prescript{(1\,2)}{3}N_r^A$} & \multicolumn{1}{|l}{$\prescript{\varepsilon}{1}N_m^A \prescript{(1\,3)}{3}N_r^A= \prescript{(1\,3)}{3}N_r^A$} & \multicolumn{1}{|l}{$\prescript{\varepsilon}{2}N_r^A \prescript{(2\,3)}{3}N_r^A= \prescript{(2\,3)}{3}N_r^A$} & \multicolumn{1}{|l}{$\prescript{\varepsilon}{2}N_m^A \prescript{(1\,2\,3)}{3}N_r^A= \prescript{(1\,2\,3)}{3}N_r^A$} & \multicolumn{1}{|l|}{$\prescript{\varepsilon}{1}N_l^A \prescript{(1\,3\,2)}{3}N_r^A= \prescript{(1\,3\,2)}{3}N_r^A$} \\\cline{2-7}
					
					& \multicolumn{1}{|l}{$\prescript{\varepsilon}{1}N_m^A \prescript{\varepsilon}{1}N_m^A= \prescript{\varepsilon}{1}N_m^A$} & \multicolumn{1}{|l}{$\prescript{\varepsilon}{2}N_m^A \prescript{(1\,2)}{1}N_m^A= \prescript{(1\,2)}{1}N_m^A$} & \multicolumn{1}{|l}{$\prescript{\varepsilon}{3}N_r^A \prescript{(1\,3)}{1}N_m^A= \prescript{(1\,3)}{1}N_m^A$} & \multicolumn{1}{|l}{$\prescript{\varepsilon}{1}N_l^A \prescript{(2\,3)}{1}N_m^A= \prescript{(2\,3)}{1}N_m^A$} & \multicolumn{1}{|l}{$\prescript{\varepsilon}{3}N_l^A \prescript{(1\,2\,3)}{1}N_m^A= \prescript{(1\,2\,3)}{1}N_m^A$} & \multicolumn{1}{|l|}{$\prescript{\varepsilon}{2}N_r^A \prescript{(1\,3\,2)}{1}N_m^A= \prescript{(1\,3\,2)}{1}N_m^A$}	\\\cline{2-7}
					
					& \multicolumn{1}{|l}{$\prescript{\varepsilon}{2}N_m^A \prescript{\varepsilon}{2}N_m^A= \prescript{\varepsilon}{2}N_m^A$} & \multicolumn{1}{|l}{$\prescript{\varepsilon}{1}N_m^A \prescript{(1\,2)}{2}N_m^A= \prescript{(1\,2)}{2}N_m^A$} & \multicolumn{1}{|l}{$\prescript{\varepsilon}{2}N_r^A \prescript{(1\,3)}{2}N_m^A= \prescript{(1\,3)}{2}N_m^A$} & \multicolumn{1}{|l}{$\prescript{\varepsilon}{3}N_l^A \prescript{(2\,3)}{2}N_m^A= \prescript{(2\,3)}{2}N_m^A$} & \multicolumn{1}{|l}{$\prescript{\varepsilon}{1}N_l^A \prescript{(1\,2\,3)}{2}N_m^A= \prescript{(1\,2\,3)}{2}N_m^A$} & \multicolumn{1}{|l|}{$\prescript{\varepsilon}{3}N_r^A \prescript{(1\,3\,2)}{2}N_m^A= \prescript{(1\,3\,2)}{2}N_m^A$}	 \\ \hline

					\multirow{6}{*}{$(1\,2)$}  & \multicolumn{1}{|l}{$\prescript{(1\,2)}{1}N_l^A \prescript{\varepsilon}{1}N_l^A= \prescript{(1\,2)}{1}N_l^A$} & \multicolumn{1}{|l}{$\prescript{(1\,2)}{2}N_r^A \prescript{(1\,2)}{1}N_l^A= \prescript{\varepsilon}{1}N_l^A$} & \multicolumn{1}{|l}{$\prescript{(1\,2)}{3}N_l^A \prescript{(1\,3)}{1}N_l^A= \prescript{(1\,2\,3)}{1}N_l^A$} & \multicolumn{1}{|l}{$\prescript{(1\,2)}{1}N_m^A \prescript{(2\,3)}{1}N_l^A= \prescript{(1\,3\,2)}{1}N_l^A$} & \multicolumn{1}{|l}{$\prescript{(1\,2)}{3}N_r^A \prescript{(1\,2\,3)}{1}N_l^A= \prescript{(1\,3)}{1}N_l^A$} & \multicolumn{1}{|l|}{$\prescript{(1\,2)}{2}N_m^A \prescript{(1\,3\,2)}{1}N_l^A= \prescript{(2\,3)}{1}N_l^A$} \\\cline{2-7}

					& \multicolumn{1}{|l}{$\prescript{(1\,2)}{3}N_l^A \prescript{\varepsilon}{3}N_l^A= \prescript{(1\,2)}{3}N_l^A$} & \multicolumn{1}{|l}{$\prescript{(1\,2)}{3}N_r^A \prescript{(1\,2)}{3}N_l^A= \prescript{\varepsilon}{3}N_l^A$} & \multicolumn{1}{|l}{$\prescript{(1\,2)}{1}N_l^A \prescript{(1\,3)}{3}N_l^A= \prescript{(1\,2\,3)}{3}N_l^A$} & \multicolumn{1}{|l}{$\prescript{(1\,2)}{2}N_m^A \prescript{(2\,3)}{3}N_l^A= \prescript{(1\,3\,2)}{3}N_l^A$} & \multicolumn{1}{|l}{$\prescript{(1\,2)}{2}N_r^A \prescript{(1\,2\,3)}{3}N_l^A= \prescript{(1\,3)}{3}N_l^A$} & \multicolumn{1}{|l|}{$\prescript{(1\,2)}{1}N_m^A \prescript{(1\,3\,2)}{3}N_l^A= \prescript{(2\,3)}{3}N_l^A$} \\\cline{2-7}

					& \multicolumn{1}{|l}{$\prescript{(1\,2)}{2}N_r^A \prescript{\varepsilon}{2}N_r^A= \prescript{(1\,2)}{2}N_r^A$} & \multicolumn{1}{|l}{$\prescript{(1\,2)}{1}N_l^A \prescript{(1\,2)}{2}N_r^A= \prescript{\varepsilon}{2}N_r^A$} & \multicolumn{1}{|l}{$\prescript{(1\,2)}{2}N_m^A \prescript{(1\,3)}{2}N_r^A= \prescript{(1\,2\,3)}{2}N_r^A$} & \multicolumn{1}{|l}{$\prescript{(1\,2)}{3}N_r^A \prescript{(2\,3)}{2}N_r^A= \prescript{(1\,3\,2)}{2}N_r^A$} & \multicolumn{1}{|l}{$\prescript{(1\,2)}{1}N_m^A \prescript{(1\,2\,3)}{2}N_r^A= \prescript{(1\,3)}{2}N_r^A$} & \multicolumn{1}{|l|}{$\prescript{(1\,2)}{3}N_l^A \prescript{(1\,3\,2)}{2}N_r^A= \prescript{(2\,3)}{2}N_r^A$} \\\cline{2-7}
					
					& \multicolumn{1}{|l}{$\prescript{(1\,2)}{3}N_r^A \prescript{\varepsilon}{3}N_r^A= \prescript{(1\,2)}{3}N_r^A$} & \multicolumn{1}{|l}{$\prescript{(1\,2)}{3}N_l^A \prescript{(1\,2)}{3}N_r^A= \prescript{\varepsilon}{3}N_r^A$} & \multicolumn{1}{|l}{$\prescript{(1\,2)}{1}N_m^A \prescript{(1\,3)}{3}N_r^A= \prescript{(1\,2\,3)}{3}N_r^A$} & \multicolumn{1}{|l}{$\prescript{(1\,2)}{2}N_r^A \prescript{(2\,3)}{3}N_r^A= \prescript{(1\,3\,2)}{3}N_r^A$} & \multicolumn{1}{|l}{$\prescript{(1\,2)}{2}N_m^A \prescript{(1\,2\,3)}{3}N_r^A= \prescript{(1\,3)}{3}N_r^A$} & \multicolumn{1}{|l|}{$\prescript{(1\,2)}{1}N_l^A \prescript{(1\,3\,2)}{3}N_r^A= \prescript{(2\,3)}{3}N_r^A$} \\\cline{2-7}
					
					& \multicolumn{1}{|l}{$\prescript{(1\,2)}{1}N_m^A \prescript{\varepsilon}{1}N_m^A= \prescript{(1\,2)}{1}N_m^A$} & \multicolumn{1}{|l}{$\prescript{(1\,2)}{2}N_m^A \prescript{(1\,2)}{1}N_m^A= \prescript{\varepsilon}{1}N_m^A$} & \multicolumn{1}{|l}{$\prescript{(1\,2)}{3}N_r^A \prescript{(1\,3)}{1}N_m^A= \prescript{(1\,2\,3)}{1}N_m^A$} & \multicolumn{1}{|l}{$\prescript{(1\,2)}{1}N_l^A \prescript{(2\,3)}{1}N_m^A= \prescript{(1\,3\,2)}{1}N_m^A$} & \multicolumn{1}{|l}{$\prescript{(1\,2)}{3}N_l^A \prescript{(1\,2\,3)}{1}N_m^A= \prescript{(1\,3)}{1}N_l^A$} & \multicolumn{1}{|l|}{$\prescript{(1\,2)}{2}N_r^A \prescript{(1\,3\,2)}{1}N_m^A= \prescript{(2\,3)}{1}N_m^A$} \\\cline{2-7}
					
					& \multicolumn{1}{|l}{$\prescript{(1\,2)}{2}N_m^A \prescript{\varepsilon}{2}N_m^A= \prescript{(1\,2)}{2}N_m^A$} & \multicolumn{1}{|l}{$\prescript{(1\,2)}{1}N_m^A \prescript{(1\,2)}{2}N_m^A= \prescript{\varepsilon}{2}N_m^A$} & \multicolumn{1}{|l}{$\prescript{(1\,2)}{2}N_r^A \prescript{(1\,3)}{2}N_m^A= \prescript{(1\,2\,3)}{2}N_m^A$} & \multicolumn{1}{|l}{$\prescript{(1\,2)}{3}N_l^A \prescript{(2\,3)}{2}N_m^A= \prescript{(1\,3\,2)}{2}N_m^A$} & \multicolumn{1}{|l}{$\prescript{(1\,2)}{1}N_l^A \prescript{(1\,2\,3)}{2}N_m^A= \prescript{(1\,3)}{2}N_l^A$} & \multicolumn{1}{|l|}{$\prescript{(1\,2)}{3}N_r^A \prescript{(1\,3\,2)}{2}N_m^A= \prescript{(2\,3)}{2}N_m^A$} \\ \hline

					\multirow{6}{*}{$(1\,3)$}  & \multicolumn{1}{|l}{$\prescript{(1\,3)}{1}N_l^A \prescript{\varepsilon}{1}N_l^A= \prescript{(1\,3)}{1}N_l^A$} & \multicolumn{1}{|l}{$\prescript{(1\,3)}{2}N_r^A \prescript{(1\,2)}{1}N_l^A= \prescript{(1\,3\,2)}{1}N_l^A$} & \multicolumn{1}{|l}{$\prescript{(1\,3)}{3}N_l^A \prescript{(1\,3)}{1}N_l^A= \prescript{\varepsilon}{1}N_l^A$} & \multicolumn{1}{|l}{$\prescript{(1\,3)}{1}N_m^A \prescript{(2\,3)}{1}N_l^A= \prescript{(1\,2\,3)}{1}N_l^A$} & \multicolumn{1}{|l}{$\prescript{(1\,3)}{3}N_r^A \prescript{(1\,2\,3)}{1}N_l^A= \prescript{(2\,3)}{1}N_l^A$} & \multicolumn{1}{|l|}{$\prescript{(1\,3)}{2}N_m^A \prescript{(1\,3\,2)}{1}N_l^A= \prescript{(1\,2)}{1}N_l^A$} \\\cline{2-7}
					
					& \multicolumn{1}{|l}{$\prescript{(1\,3)}{3}N_l^A \prescript{\varepsilon}{3}N_l^A= \prescript{(1\,3)}{3}N_l^A$} & \multicolumn{1}{|l}{$\prescript{(1\,3)}{3}N_r^A \prescript{(1\,2)}{3}N_l^A= \prescript{(1\,3\,2)}{3}N_l^A$} & \multicolumn{1}{|l}{$\prescript{(1\,3)}{1}N_l^A \prescript{(1\,3)}{3}N_l^A= \prescript{\varepsilon}{3}N_l^A$} & \multicolumn{1}{|l}{$\prescript{(1\,3)}{2}N_m^A \prescript{(2\,3)}{3}N_l^A= \prescript{(1\,2\,3)}{3}N_l^A$} & \multicolumn{1}{|l}{$\prescript{(1\,3)}{2}N_r^A \prescript{(1\,2\,3)}{3}N_l^A= \prescript{(2\,3)}{3}N_l^A$} & \multicolumn{1}{|l|}{$\prescript{(1\,3)}{1}N_m^A \prescript{(1\,3\,2)}{3}N_l^A= \prescript{(1\,2)}{3}N_l^A$} \\\cline{2-7}
					
					& \multicolumn{1}{|l}{$\prescript{(1\,3)}{2}N_r^A \prescript{\varepsilon}{2}N_r^A= \prescript{(1\,3)}{2}N_r^A$} & \multicolumn{1}{|l}{$\prescript{(1\,3)}{1}N_l^A \prescript{(1\,2)}{2}N_r^A= \prescript{(1\,3\,2)}{2}N_r^A$} & \multicolumn{1}{|l}{$\prescript{(1\,3)}{2}N_m^A \prescript{(1\,3)}{2}N_r^A= \prescript{\varepsilon}{2}N_r^A$} & \multicolumn{1}{|l}{$\prescript{(1\,3)}{3}N_r^A \prescript{(2\,3)}{2}N_r^A= \prescript{(1\,2\,3)}{2}N_r^A$} & \multicolumn{1}{|l}{$\prescript{(1\,3)}{1}N_m^A \prescript{(1\,2\,3)}{2}N_r^A= \prescript{(2\,3)}{2}N_r^A$} & \multicolumn{1}{|l|}{$\prescript{(1\,3)}{3}N_l^A \prescript{(1\,3\,2)}{2}N_r^A= \prescript{(1\,2)}{2}N_r^A$} \\\cline{2-7}
					
					& \multicolumn{1}{|l}{$\prescript{(1\,3)}{3}N_r^A \prescript{\varepsilon}{3}N_r^A= \prescript{(1\,3)}{3}N_r^A$} & \multicolumn{1}{|l}{$\prescript{(1\,3)}{3}N_l^A \prescript{(1\,2)}{3}N_r^A= \prescript{(1\,3\,2)}{3}N_r^A$} & \multicolumn{1}{|l}{$\prescript{(1\,3)}{1}N_m^A \prescript{(1\,3)}{3}N_r^A= \prescript{\varepsilon}{3}N_r^A$} & \multicolumn{1}{|l}{$\prescript{(1\,3)}{2}N_r^A \prescript{(2\,3)}{3}N_r^A= \prescript{(1\,2\,3)}{3}N_r^A$} & \multicolumn{1}{|l}{$\prescript{(1\,3)}{2}N_m^A \prescript{(1\,2\,3)}{3}N_r^A= \prescript{(2\,3)}{3}N_r^A$} & \multicolumn{1}{|l|}{$\prescript{(1\,3)}{1}N_l^A \prescript{(1\,3\,2)}{3}N_r^A= \prescript{(1\,2)}{3}N_r^A$} \\\cline{2-7}
					
					& \multicolumn{1}{|l}{$\prescript{(1\,3)}{1}N_m^A \prescript{\varepsilon}{1}N_m^A= \prescript{(1\,3)}{1}N_m^A$} & \multicolumn{1}{|l}{$\prescript{(1\,3)}{2}N_m^A \prescript{(1\,2)}{1}N_m^A= \prescript{(1\,3\,2)}{1}N_m^A$} & \multicolumn{1}{|l}{$\prescript{(1\,3)}{3}N_r^A \prescript{(1\,3)}{1}N_m^A= \prescript{\varepsilon}{1}N_m^A$} & \multicolumn{1}{|l}{$\prescript{(1\,3)}{1}N_l^A \prescript{(2\,3)}{1}N_m^A= \prescript{(1\,2\,3)}{1}N_m^A$} & \multicolumn{1}{|l}{$\prescript{(1\,3)}{3}N_l^A \prescript{(1\,2\,3)}{1}N_m^A= \prescript{(2\,3)}{1}N_m^A$} & \multicolumn{1}{|l|}{$\prescript{(1\,3)}{2}N_r^A \prescript{(1\,3\,2)}{1}N_m^A= \prescript{(1\,2)}{1}N_m^A$} \\\cline{2-7}
					
					& \multicolumn{1}{|l}{$\prescript{(1\,3)}{2}N_m^A \prescript{\varepsilon}{2}N_m^A= \prescript{(1\,3)}{2}N_m^A$} & \multicolumn{1}{|l}{$\prescript{(1\,3)}{1}N_m^A \prescript{(1\,2)}{2}N_m^A= \prescript{(1\,3\,2)}{2}N_m^A$} & \multicolumn{1}{|l}{$\prescript{(1\,3)}{2}N_r^A \prescript{(1\,3)}{2}N_m^A= \prescript{\varepsilon}{2}N_m^A$} & \multicolumn{1}{|l}{$\prescript{(1\,3)}{3}N_l^A \prescript{(2\,3)}{2}N_m^A= \prescript{(1\,2\,3)}{2}N_m^A$} & \multicolumn{1}{|l}{$\prescript{(1\,3)}{1}N_l^A \prescript{(1\,2\,3)}{2}N_m^A= \prescript{(2\,3)}{2}N_m^A$} & \multicolumn{1}{|l|}{$\prescript{(1\,3)}{3}N_r^A \prescript{(1\,3\,2)}{2}N_m^A= \prescript{(1\,2)}{2}N_m^A$} \\ \hline

					\multirow{6}{*}{$(2\,3)$}  & \multicolumn{1}{|l}{$\prescript{(2\,3)}{1}N_l^A \prescript{\varepsilon}{1}N_l^A= \prescript{(2\,3)}{1}N_l^A$} & \multicolumn{1}{|l}{$\prescript{(2\,3)}{2}N_r^A \prescript{(1\,2)}{1}N_l^A= \prescript{(1\,2\,3)}{1}N_l^A$} & \multicolumn{1}{|l}{$\prescript{(2\,3)}{3}N_l^A \prescript{(1\,3)}{1}N_l^A= \prescript{(1\,3\,2)}{1}N_l^A$} & \multicolumn{1}{|l}{$\prescript{(2\,3)}{1}N_m^A \prescript{(2\,3)}{1}N_l^A= \prescript{\varepsilon}{1}N_l^A$} & \multicolumn{1}{|l}{$\prescript{(2\,3)}{3}N_r^A \prescript{(1\,2\,3)}{1}N_l^A= \prescript{(1\,2)}{1}N_l^A$} & \multicolumn{1}{|l|}{$\prescript{(2\,3)}{2}N_m^A \prescript{(1\,3\,2)}{1}N_l^A= \prescript{(1\,3)}{1}N_l^A$} \\\cline{2-7}
					
					& \multicolumn{1}{|l}{$\prescript{(2\,3)}{3}N_l^A \prescript{\varepsilon}{3}N_l^A= \prescript{(2\,3)}{3}N_l^A$} & \multicolumn{1}{|l}{$\prescript{(2\,3)}{3}N_r^A \prescript{(1\,2)}{3}N_l^A= \prescript{(1\,2\,3)}{3}N_l^A$} & \multicolumn{1}{|l}{$\prescript{(2\,3)}{1}N_l^A \prescript{(1\,3)}{3}N_l^A= \prescript{(1\,3\,2)}{3}N_l^A$} & \multicolumn{1}{|l}{$\prescript{(2\,3)}{2}N_m^A \prescript{(2\,3)}{3}N_l^A= \prescript{\varepsilon}{3}N_l^A$} & \multicolumn{1}{|l}{$\prescript{(2\,3)}{2}N_r^A \prescript{(1\,2\,3)}{3}N_l^A= \prescript{(1\,2)}{3}N_l^A$} & \multicolumn{1}{|l|}{$\prescript{(2\,3)}{1}N_m^A \prescript{(1\,3\,2)}{3}N_l^A= \prescript{(1\,3)}{3}N_l^A$} \\\cline{2-7}
					
					& \multicolumn{1}{|l}{$\prescript{(2\,3)}{2}N_r^A \prescript{\varepsilon}{2}N_r^A= \prescript{(2\,3)}{2}N_r^A$} & \multicolumn{1}{|l}{$\prescript{(2\,3)}{1}N_l^A \prescript{(1\,2)}{2}N_r^A= \prescript{(1\,2\,3)}{2}N_r^A$} & \multicolumn{1}{|l}{$\prescript{(2\,3)}{2}N_m^A \prescript{(1\,3)}{2}N_r^A= \prescript{(1\,3\,2)}{2}N_r^A$} & \multicolumn{1}{|l}{$\prescript{(2\,3)}{3}N_r^A \prescript{(2\,3)}{2}N_r^A= \prescript{\varepsilon}{2}N_r^A$} & \multicolumn{1}{|l}{$\prescript{(2\,3)}{1}N_m^A \prescript{(1\,2\,3)}{2}N_r^A= \prescript{(1\,2)}{2}N_r^A$} & \multicolumn{1}{|l|}{$\prescript{(2\,3)}{3}N_l^A \prescript{(1\,3\,2)}{2}N_r^A= \prescript{(1\,3)}{2}N_r^A$} \\\cline{2-7}
					
					& \multicolumn{1}{|l}{$\prescript{(2\,3)}{3}N_r^A \prescript{\varepsilon}{3}N_r^A= \prescript{(2\,3)}{3}N_r^A$} & \multicolumn{1}{|l}{$\prescript{(2\,3)}{3}N_l^A \prescript{(1\,2)}{3}N_r^A= \prescript{(1\,2\,3)}{3}N_r^A$} & \multicolumn{1}{|l}{$\prescript{(2\,3)}{1}N_m^A \prescript{(1\,3)}{3}N_r^A= \prescript{(1\,3\,2)}{3}N_r^A$} & \multicolumn{1}{|l}{$\prescript{(2\,3)}{2}N_r^A \prescript{(2\,3)}{3}N_r^A= \prescript{\varepsilon}{3}N_r^A$} & \multicolumn{1}{|l}{$\prescript{(2\,3)}{2}N_m^A \prescript{(1\,2\,3)}{3}N_r^A= \prescript{(1\,2)}{3}N_r^A$} & \multicolumn{1}{|l|}{$\prescript{(2\,3)}{1}N_l^A \prescript{(1\,3\,2)}{3}N_r^A= \prescript{(1\,3)}{3}N_r^A$} \\\cline{2-7}
					
					& \multicolumn{1}{|l}{$\prescript{(2\,3)}{1}N_m^A \prescript{\varepsilon}{1}N_m^A= \prescript{(2\,3)}{1}N_m^A$} & \multicolumn{1}{|l}{$\prescript{(2\,3)}{2}N_m^A \prescript{(1\,2)}{1}N_m^A= \prescript{(1\,2\,3)}{1}N_m^A$} & \multicolumn{1}{|l}{$\prescript{(2\,3)}{3}N_r^A \prescript{(1\,3)}{1}N_m^A= \prescript{(1\,3\,2)}{1}N_m^A$} & \multicolumn{1}{|l}{$\prescript{(2\,3)}{1}N_l^A \prescript{(2\,3)}{1}N_m^A= \prescript{\varepsilon}{1}N_m^A$} & \multicolumn{1}{|l}{$\prescript{(2\,3)}{3}N_l^A \prescript{(1\,2\,3)}{1}N_m^A= \prescript{(1\,2)}{1}N_m^A$} & \multicolumn{1}{|l|}{$\prescript{(2\,3)}{2}N_r^A \prescript{(1\,3\,2)}{1}N_m^A= \prescript{(1\,3)}{1}N_m^A$} \\\cline{2-7}
					
					& \multicolumn{1}{|l}{$\prescript{(2\,3)}{2}N_m^A \prescript{\varepsilon}{2}N_m^A= \prescript{(2\,3)}{2}N_m^A$} & \multicolumn{1}{|l}{$\prescript{(2\,3)}{1}N_m^A \prescript{(1\,2)}{2}N_m^A= \prescript{(1\,2\,3)}{2}N_m^A$} & \multicolumn{1}{|l}{$\prescript{(2\,3)}{2}N_r^A \prescript{(1\,3)}{2}N_m^A= \prescript{(1\,3\,2)}{2}N_m^A$} & \multicolumn{1}{|l}{$\prescript{(2\,3)}{3}N_l^A \prescript{(2\,3)}{2}N_m^A= \prescript{\varepsilon}{2}N_m^A$} & \multicolumn{1}{|l}{$\prescript{(2\,3)}{1}N_l^A \prescript{(1\,2\,3)}{2}N_m^A= \prescript{(1\,2)}{2}N_m^A$} & \multicolumn{1}{|l|}{$\prescript{(2\,3)}{3}N_r^A \prescript{(1\,3\,2)}{2}N_m^A= \prescript{(1\,3)}{2}N_m^A$} \\ \hline

					\multirow{6}{*}{$(1\,2\,3)$}  & \multicolumn{1}{|l}{$\prescript{(1\,2\,3)}{1}N_l^A \prescript{\varepsilon}{1}N_l^A= \prescript{(1\,2\,3)}{1}N_l^A$} & \multicolumn{1}{|l}{$\prescript{(1\,2\,3)}{2}N_r^A \prescript{(1\,2)}{1}N_l^A= \prescript{(2\,3)}{1}N_l^A$} & \multicolumn{1}{|l}{$\prescript{(1\,2\,3)}{3}N_l^A \prescript{(1\,3)}{1}N_l^A= \prescript{(1\,2)}{1}N_l^A$} & \multicolumn{1}{|l}{$\prescript{(1\,2\,3)}{1}N_m^A \prescript{(2\,3)}{1}N_l^A= \prescript{(1\,3)}{1}N_l^A$} & \multicolumn{1}{|l}{$\prescript{(1\,2\,3)}{3}N_r^A \prescript{(1\,2\,3)}{1}N_l^A= \prescript{(1\,3\,2)}{1}N_l^A$} & \multicolumn{1}{|l|}{$\prescript{(1\,2\,3)}{2}N_m^A \prescript{(1\,3\,2)}{1}N_l^A= \prescript{\varepsilon}{1}N_l^A$} \\\cline{2-7}

					& \multicolumn{1}{|l}{$\prescript{(1\,2\,3)}{3}N_l^A \prescript{\varepsilon}{3}N_l^A= \prescript{(1\,2\,3)}{3}N_l^A$} & \multicolumn{1}{|l}{$\prescript{(1\,2\,3)}{3}N_r^A \prescript{(1\,2)}{3}N_l^A= \prescript{(2\,3)}{3}N_l^A$} & \multicolumn{1}{|l}{$\prescript{(1\,2\,3)}{1}N_l^A \prescript{(1\,3)}{3}N_l^A= \prescript{(1\,2)}{3}N_l^A$} & \multicolumn{1}{|l}{$\prescript{(1\,2\,3)}{2}N_m^A \prescript{(2\,3)}{3}N_l^A= \prescript{(1\,3)}{3}N_l^A$} & \multicolumn{1}{|l}{$\prescript{(1\,2\,3)}{2}N_r^A \prescript{(1\,2\,3)}{3}N_l^A= \prescript{(1\,3\,2)}{3}N_l^A$} & \multicolumn{1}{|l|}{$\prescript{(1\,2\,3)}{1}N_m^A \prescript{(1\,3\,2)}{3}N_l^A= \prescript{\varepsilon}{3}N_l^A$} \\\cline{2-7}
					
					& \multicolumn{1}{|l}{$\prescript{(1\,2\,3)}{2}N_r^A \prescript{\varepsilon}{2}N_r^A= \prescript{(1\,2\,3)}{2}N_r^A$} & \multicolumn{1}{|l}{$\prescript{(1\,2\,3)}{1}N_l^A \prescript{(1\,2)}{2}N_r^A= \prescript{(2\,3)}{2}N_r^A$} & \multicolumn{1}{|l}{$\prescript{(1\,2\,3)}{2}N_m^A \prescript{(1\,3)}{2}N_r^A= \prescript{(1\,2)}{2}N_r^A$} & \multicolumn{1}{|l}{$\prescript{(1\,2\,3)}{3}N_r^A \prescript{(2\,3)}{2}N_r^A= \prescript{(1\,3)}{2}N_r^A$} & \multicolumn{1}{|l}{$\prescript{(1\,2\,3)}{1}N_m^A \prescript{(1\,2\,3)}{2}N_r^A= \prescript{(1\,3\,2)}{2}N_r^A$} & \multicolumn{1}{|l|}{$\prescript{(1\,2\,3)}{3}N_l^A \prescript{(1\,3\,2)}{2}N_r^A= \prescript{\varepsilon}{2}N_r^A$} \\\cline{2-7}
					
					& \multicolumn{1}{|l}{$\prescript{(1\,2\,3)}{3}N_r^A \prescript{\varepsilon}{3}N_r^A= \prescript{(1\,2\,3)}{3}N_r^A$} & \multicolumn{1}{|l}{$\prescript{(1\,2\,3)}{3}N_l^A \prescript{(1\,2)}{3}N_r^A= \prescript{(2\,3)}{3}N_r^A$} & \multicolumn{1}{|l}{$\prescript{(1\,2\,3)}{1}N_m^A \prescript{(1\,3)}{3}N_r^A= \prescript{(1\,2)}{3}N_r^A$} & \multicolumn{1}{|l}{$\prescript{(1\,2\,3)}{2}N_r^A \prescript{(2\,3)}{3}N_r^A= \prescript{(1\,3)}{3}N_r^A$} & \multicolumn{1}{|l}{$\prescript{(1\,2\,3)}{2}N_m^A \prescript{(1\,2\,3)}{3}N_r^A= \prescript{(1\,3\,2)}{3}N_r^A$} & \multicolumn{1}{|l|}{$\prescript{(1\,2\,3)}{1}N_l^A \prescript{(1\,3\,2)}{3}N_r^A= \prescript{\varepsilon}{3}N_r^A$} \\\cline{2-7}
					
					& \multicolumn{1}{|l}{$\prescript{(1\,2\,3)}{1}N_m^A \prescript{\varepsilon}{1}N_m^A= \prescript{(1\,2\,3)}{1}N_m^A$} & \multicolumn{1}{|l}{$\prescript{(1\,2\,3)}{2}N_m^A \prescript{(1\,2)}{1}N_m^A= \prescript{(2\,3)}{1}N_m^A$} & \multicolumn{1}{|l}{$\prescript{(1\,2\,3)}{3}N_r^A \prescript{(1\,3)}{1}N_m^A= \prescript{(1\,2)}{1}N_m^A$} & \multicolumn{1}{|l}{$\prescript{(1\,2\,3)}{1}N_l^A \prescript{(2\,3)}{1}N_m^A= \prescript{(1\,3)}{1}N_m^A$} & \multicolumn{1}{|l}{$\prescript{(1\,2\,3)}{3}N_l^A \prescript{(1\,2\,3)}{1}N_m^A= \prescript{(1\,3\,2)}{1}N_m^A$} & \multicolumn{1}{|l|}{$\prescript{(1\,2\,3)}{2}N_r^A \prescript{(1\,3\,2)}{1}N_m^A= \prescript{\varepsilon}{1}N_m^A$} \\\cline{2-7}
					
					& \multicolumn{1}{|l}{$\prescript{(1\,2\,3)}{2}N_m^A \prescript{\varepsilon}{2}N_m^A= \prescript{(1\,2\,3)}{2}N_m^A$} & \multicolumn{1}{|l}{$\prescript{(1\,2\,3)}{1}N_m^A \prescript{(1\,2)}{2}N_m^A= \prescript{(2\,3)}{2}N_m^A$} & \multicolumn{1}{|l}{$\prescript{(1\,2\,3)}{2}N_r^A \prescript{(1\,3)}{2}N_m^A= \prescript{(1\,2)}{2}N_m^A$} & \multicolumn{1}{|l}{$\prescript{(1\,2\,3)}{3}N_l^A \prescript{(2\,3)}{2}N_m^A= \prescript{(1\,3)}{2}N_m^A$} & \multicolumn{1}{|l}{$\prescript{(1\,2\,3)}{1}N_l^A \prescript{(1\,2\,3)}{2}N_m^A= \prescript{(1\,3\,2)}{2}N_m^A$} & \multicolumn{1}{|l|}{$\prescript{(1\,2\,3)}{3}N_r^A \prescript{(1\,3\,2)}{2}N_m^A= \prescript{\varepsilon}{2}N_m^A$} \\ \hline
					
					\multirow{6}{*}{$(1\,3\,2)$}  & \multicolumn{1}{|l}{$\prescript{(1\,3\,2)}{1}N_l^A \prescript{\varepsilon}{1}N_l^A= \prescript{(1\,3\,2)}{1}N_l^A$} & \multicolumn{1}{|l}{$\prescript{(1\,3\,2)}{2}N_r^A \prescript{(1\,2)}{1}N_l^A= \prescript{(1\,3)}{1}N_l^A$} & \multicolumn{1}{|l}{$\prescript{(1\,3\,2)}{3}N_l^A \prescript{(1\,3)}{1}N_l^A= \prescript{(2\,3)}{1}N_l^A$} & \multicolumn{1}{|l}{$\prescript{(1\,3\,2)}{1}N_m^A \prescript{(2\,3)}{1}N_l^A= \prescript{(1\,2)}{1}N_l^A$} & \multicolumn{1}{|l}{$\prescript{(1\,3\,2)}{3}N_r^A \prescript{(1\,2\,3)}{1}N_l^A= \prescript{\varepsilon}{1}N_l^A$} & \multicolumn{1}{|l|}{$\prescript{(1\,3\,2)}{2}N_m^A \prescript{(1\,3\,2)}{1}N_l^A= \prescript{(1\,2\,3)}{1}N_l^A$} \\\cline{2-7}

					& \multicolumn{1}{|l}{$\prescript{(1\,3\,2)}{3}N_l^A \prescript{\varepsilon}{3}N_l^A= \prescript{(1\,3\,2)}{3}N_l^A$} & \multicolumn{1}{|l}{$\prescript{(1\,3\,2)}{3}N_r^A \prescript{(1\,2)}{3}N_l^A= \prescript{(1\,3)}{3}N_l^A$} & \multicolumn{1}{|l}{$\prescript{(1\,3\,2)}{1}N_l^A \prescript{(1\,3)}{3}N_l^A= \prescript{(2\,3)}{3}N_l^A$} & \multicolumn{1}{|l}{$\prescript{(1\,3\,2)}{2}N_m^A \prescript{(2\,3)}{3}N_l^A= \prescript{(1\,2)}{3}N_l^A$} & \multicolumn{1}{|l}{$\prescript{(1\,3\,2)}{2}N_r^A \prescript{(1\,2\,3)}{3}N_l^A= \prescript{\varepsilon}{3}N_l^A$} & \multicolumn{1}{|l|}{$\prescript{(1\,3\,2)}{1}N_m^A \prescript{(1\,3\,2)}{3}N_l^A= \prescript{(1\,2\,3)}{3}N_l^A$} \\\cline{2-7}
					
					& \multicolumn{1}{|l}{$\prescript{(1\,3\,2)}{2}N_r^A \prescript{\varepsilon}{2}N_r^A= \prescript{(1\,3\,2)}{2}N_r^A$} & \multicolumn{1}{|l}{$\prescript{(1\,3\,2)}{1}N_l^A \prescript{(1\,2)}{2}N_r^A= \prescript{(1\,3)}{2}N_r^A$} & \multicolumn{1}{|l}{$\prescript{(1\,3\,2)}{2}N_m^A \prescript{(1\,3)}{2}N_r^A= \prescript{(2\,3)}{2}N_r^A$} & \multicolumn{1}{|l}{$\prescript{(1\,3\,2)}{3}N_r^A \prescript{(2\,3)}{2}N_r^A= \prescript{(1\,2)}{2}N_r^A$} & \multicolumn{1}{|l}{$\prescript{(1\,3\,2)}{1}N_m^A \prescript{(1\,2\,3)}{2}N_r^A= \prescript{\varepsilon}{2}N_r^A$} & \multicolumn{1}{|l|}{$\prescript{(1\,3\,2)}{3}N_l^A \prescript{(1\,3\,2)}{2}N_r^A= \prescript{(1\,2\,3)}{2}N_r^A$} \\\cline{2-7}
					
					& \multicolumn{1}{|l}{$\prescript{(1\,3\,2)}{3}N_r^A \prescript{\varepsilon}{3}N_r^A= \prescript{(1\,3\,2)}{3}N_r^A$} & \multicolumn{1}{|l}{$\prescript{(1\,3\,2)}{3}N_l^A \prescript{(1\,2)}{3}N_r^A= \prescript{(1\,3)}{3}N_r^A$} & \multicolumn{1}{|l}{$\prescript{(1\,3\,2)}{1}N_m^A \prescript{(1\,3)}{3}N_r^A= \prescript{(2\,3)}{3}N_r^A$} & \multicolumn{1}{|l}{$\prescript{(1\,3\,2)}{2}N_r^A \prescript{(2\,3)}{3}N_r^A= \prescript{(1\,2)}{3}N_r^A$} & \multicolumn{1}{|l}{$\prescript{(1\,3\,2)}{2}N_m^A \prescript{(1\,2\,3)}{3}N_r^A= \prescript{\varepsilon}{3}N_r^A$} & \multicolumn{1}{|l|}{$\prescript{(1\,3\,2)}{1}N_l^A \prescript{(1\,3\,2)}{3}N_r^A= \prescript{(1\,2\,3)}{3}N_r^A$} \\\cline{2-7}
					
					& \multicolumn{1}{|l}{$\prescript{(1\,3\,2)}{1}N_m^A \prescript{\varepsilon}{1}N_m^A= \prescript{(1\,3\,2)}{1}N_m^A$} & \multicolumn{1}{|l}{$\prescript{(1\,3\,2)}{2}N_m^A \prescript{(1\,2)}{1}N_m^A= \prescript{(1\,3)}{1}N_m^A$} & \multicolumn{1}{|l}{$\prescript{(1\,3\,2)}{3}N_r^A \prescript{(1\,3)}{1}N_m^A= \prescript{(2\,3)}{1}N_m^A$} & \multicolumn{1}{|l}{$\prescript{(1\,3\,2)}{1}N_l^A \prescript{(2\,3)}{1}N_m^A= \prescript{(1\,2)}{1}N_m^A$} & \multicolumn{1}{|l}{$\prescript{(1\,3\,2)}{3}N_l^A \prescript{(1\,2\,3)}{1}N_m^A= \prescript{\varepsilon}{1}N_m^A$} & \multicolumn{1}{|l|}{$\prescript{(1\,3\,2)}{2}N_r^A \prescript{(1\,3\,2)}{1}N_m^A= \prescript{(1\,2\,3)}{1}N_m^A$} \\\cline{2-7}
					
					& \multicolumn{1}{|l}{$\prescript{(1\,3\,2)}{2}N_m^A \prescript{\varepsilon}{2}N_m^A= \prescript{(1\,3\,2)}{2}N_m^A$} & \multicolumn{1}{|l}{$\prescript{(1\,3\,2)}{1}N_m^A \prescript{(1\,2)}{2}N_m^A= \prescript{(1\,3)}{2}N_m^A$} & \multicolumn{1}{|l}{$\prescript{(1\,3\,2)}{2}N_r^A \prescript{(1\,3)}{2}N_m^A= \prescript{(2\,3)}{2}N_m^A$} & \multicolumn{1}{|l}{$\prescript{(1\,3\,2)}{3}N_l^A \prescript{(2\,3)}{2}N_m^A= \prescript{(1\,2)}{2}N_m^A$} & \multicolumn{1}{|l}{$\prescript{(1\,3\,2)}{1}N_l^A \prescript{(1\,2\,3)}{2}N_m^A= \prescript{\varepsilon}{2}N_m^A$} & \multicolumn{1}{|l|}{$\prescript{(1\,3\,2)}{3}N_r^A \prescript{(1\,3\,2)}{2}N_m^A= \prescript{(1\,2\,3)}{2}N_m^A$} \\ \hline
				\end{tabular}
			}
		\end{table}
	\end{landscape}
	\clearpage
}

Let quasigroup $(Q,\circ)$ be an isostrophic image of a quasigroup $(Q,\ast)$ with an isostrophism $\theta$, i.e., $(Q,\circ)=(Q,\ast)\theta$.
In order to distinguish the $\sigma$-A-nuclei of the quasigroups $(Q,\ast)$ and $(Q,\circ)$, we shall replace the symbol $A$ (which stands for autotopy) by $\ast$ and $\circ$ (the binary operations) respectively.

The following theorem gives a complete description of left, right and middle $\sigma$-A-nuclei of isostrophic images of the form $(Q,\circ)=(Q,\ast)(\tau,(\alpha_1,\alpha_2,\alpha_3))$ of a quasigroup $(Q,\ast)$ and their respective component sets, for any $\tau \in S_3$.



\begin{theorem} \label{nucleus_iso}
	If quasigroup $(Q,\circ)$ is an isostrophic image of a quasigroup $(Q,\ast)$ with an isostrophy $\theta=(\tau,(\alpha_1,\alpha_2,\alpha_3))$, i.e., $(Q,\circ)=(Q,\ast)\theta$, where $\tau \in S_3$, then the following hold:
	\begin{enumerate}
		\item $\prescript{\sigma}{}N_l^\circ= \theta^{-1} \prescript{\tau \sigma \tau^{-1}}{}N_v^\ast \theta$ $\iff$ $\sigma=\varepsilon$ or $(1\,3)$, where $v=
		\begin{cases*}
		r & if $\tau^{-1} 2=1$,\\ 
		l & if $\tau^{-1} 2=2$, \\
		m & if $\tau^{-1} 2=3.$
		\end{cases*}$
		Further, 
		if $\sigma \in \{\varepsilon, (1\,3)\}$, then
		$\prescript{\sigma}{1}N_l^\circ= \alpha_{\sigma^{-1}1}^{-1} \prescript{\tau \sigma \tau^{-1}}{\tau^{-1}1}N_v^\ast \alpha_1$ and $\prescript{\sigma}{3}N_l^\circ= \alpha_{\sigma^{-1}3}^{-1} \prescript{\tau \sigma \tau^{-1}}{\tau^{-1}3}N_v^\ast \alpha_3$. 
		
		\item $\prescript{\sigma}{}N_r^\circ= \theta^{-1} \prescript{\tau \sigma \tau^{-1}}{}N_v^\ast \theta$ $\iff$ $\sigma=\varepsilon$ or $(2\,3)$, where $v=
		\begin{cases*}
		r & if $\tau^{-1} 1=1$,\\ 
		l & if $\tau^{-1} 1=2$, \\
		m & if $\tau^{-1} 1=3.$
		\end{cases*}$
		Further, 
		if $\sigma \in \{\varepsilon, (2\,3)\}$ then $\prescript{\sigma}{2}N_r^\circ= \alpha_{\sigma^{-1}2}^{-1} \prescript{\tau \sigma  \tau^{-1}}{\tau^{-1}2}N_v^\ast \alpha_2$ and $\prescript{\sigma}{3}N_r^\circ= \alpha_{\sigma^{-1}3}^{-1} \prescript{\tau \sigma  \tau^{-1}}{\tau^{-1}3}N_v^\ast \alpha_3$. 
		
		\item $\prescript{\sigma}{}N_m^\circ= \theta^{-1} \prescript{\tau \sigma \tau^{-1}}{}N_v^\ast \theta$ $\iff$ $\sigma=\varepsilon$ or $(1\,2)$, where $v=
		\begin{cases*}
		r & if $\tau^{-1} 3=1$,\\ 
		l & if $\tau^{-1} 3=2$, \\
		m & if $\tau^{-1} 3=3.$
		\end{cases*}$
		Further, 
		if $\sigma \in \{\varepsilon,(1\,2)\}$ then $\prescript{\sigma}{1}N_m^\circ= \alpha_{\sigma^{-1}1}^{-1} \prescript{\tau \sigma \tau^{-1}}{\tau^{-1}1}N_v^\ast \alpha_1$ and $\prescript{\sigma}{2}N_m^\circ= \alpha_{\sigma^{-1}2}^{-1} \prescript{\tau \sigma \tau^{-1}}{\tau^{-1}2}N_v^\ast \alpha_2$. 	
	\end{enumerate} 
\end{theorem}

\begin{proof}
	$(1)$ We shall show the inclusions in both directions. 
	Let $T=(\alpha_1,\alpha_2,\alpha_3)$, i.e., $\theta=(\tau,T)$ and $\varphi \in \prescript{\sigma}{}N_l^\circ$. 
	Then there exist permutations $\varphi_1,\varphi_2,\varphi_3$ of $Q$ such that $\varphi= (\sigma, (\varphi_1,\varphi_2,\varphi_3))=(\sigma,R)$ (say) where $\varphi_2=\varepsilon$. 
	We have $\varphi\in Aus(Q,\circ)$ and by Theorem \ref{thm1}, $\theta \varphi\theta^{-1} \in Aus(Q,\ast)$.
	Also 
	\begin{align*}\theta \varphi\theta^{-1}&= (\tau,T) (\sigma, R) (\tau, T)^{-1}\\
	&= (\tau, T) (\sigma \tau ^{-1}, R^{\tau^{-1}}(T^{-1})^{\tau^{-1}} )\\
	&= (\tau \sigma \tau ^{-1}, T^{\sigma \tau ^{-1}} R^{\tau^{-1}}(T^{-1})^{\tau^{-1}} ).
	\end{align*}
	On substituting the values of $T$ and $R$ we get 
	\begin{align*}
	\theta \varphi\theta^{-1}&= (\tau \sigma \tau ^{-1}, (\alpha_{( \tau \sigma^{-1})1} \varphi_{\tau 1} \alpha^{-1}_{\tau 1},\alpha_{( \tau \sigma^{-1})2} \varphi_{\tau 2} \alpha^{-1}_{\tau 2}, \alpha_{( \tau \sigma^{-1})3} \varphi_{\tau 3} \alpha^{-1}_{\tau 3}))\\ &=(\tau \sigma \tau ^{-1}, (\beta_1,\beta_2,\beta_3))\text{ (say)}.
	\end{align*}
	Comparing the components we obtain $\beta_i= \alpha_{( \tau \sigma^{-1})i} \varphi_{\tau i} \alpha^{-1}_{\tau i}$, i.e., $\beta_{\tau^{-1}i}= \alpha_{\sigma^{-1}i} \varphi_{i} \alpha^{-1}_{i}$ for $i=1,2,3$. 
	As $\varphi_{2}=\varepsilon$, for $i=2$ we have  $\beta_{\tau^{-1}2}= \alpha_{\sigma^{-1}2} \alpha^{-1}_{2}$.
	Thus $\theta \varphi\theta^{-1} \in \prescript{\tau \sigma \tau^{-1}}{}N_v^\ast $ iff $\beta_{\tau^{-1}2}=\varepsilon$ iff $\alpha_{\sigma^{-1}2} \alpha^{-1}_{2}=\varepsilon$, i.e., $\sigma=\varepsilon$ or $(1\,3)$ (since $\sigma \in S_3$).
	Hence 
	\begin{equation} \label{eq1}
	\theta \prescript{\sigma}{}N_l^\circ \theta^{-1} \subseteq \prescript{\tau \sigma \tau^{-1}}{}N_v^\ast \iff \sigma=\varepsilon \;\text{or}\; (1\,3).
	\end{equation}
	Also, for $\sigma \in \{\varepsilon, (1\,3)\}$ and $i=1,3$, we obtain $\alpha_{\sigma^{-1}i} \varphi_{i} \alpha^{-1}_{i}= \beta_{\tau^{-1}i}\in \prescript{\tau \sigma \tau^{-1}}{i}N_v^\ast $, that gives
	$\alpha_{\sigma^{-1}1} \prescript{\sigma}{1}N_l^\circ \alpha_1^{-1} \subseteq \prescript{\tau \sigma \tau^{-1}}{\tau^{-1}1}N_v^\ast $ and $\alpha_{\sigma^{-1}3} \prescript{\sigma}{3}N_l^\circ \alpha_3^{-1} \subseteq \prescript{\tau \sigma \tau^{-1}}{\tau^{-1}3}N_v^\ast $.
	
	Conversely, let $\varphi=(\tau \sigma \tau^{-1},R) \in \prescript{\tau \sigma \tau^{-1}}{}N_v^\ast $.
	Then there exist permutations $\varphi_1,\varphi_2,\varphi_3$ of $Q$ such that $R=(\varphi_1, \varphi_2 , \varphi_3)$,  where $\varphi_{\tau^{-1} 2}=\varepsilon$.
	We have $\varphi\in Aus(Q,\ast)$, and by Theorem \ref{thm1}, we get $\theta^{-1}\varphi\theta \in Aus(Q,\circ)$.
	Also \begin{align*}
	\theta^{-1}\varphi\theta &= (\tau,T)^{-1} (\tau \sigma \tau^{-1},R) (\tau, T)\\
	&= (\tau^{-1}, (T^{-1})^{\tau^{-1}}) (\tau \sigma, R^{\tau}T )\\
	&= (\sigma, (T^{-1})^{\sigma}R^{\tau}T).
	\end{align*}
	On substituting the values of $T$ and $R$, and as $\varphi_{\tau^{-1} 2}=\varepsilon$, we get $\theta^{-1}\varphi\theta = (\sigma, (\alpha^{-1}_{\sigma^{-1}1} \varphi_{\tau^{-1}1} \alpha_1, \alpha^{-1}_{\sigma^{-1}2} \alpha_2, \alpha^{-1}_{\sigma^{-1}3} \varphi_{\tau^{-1}3} \alpha_3)) \in \prescript{\sigma}{}N_l^\circ $ iff $\alpha_{\sigma^{-1}2} \alpha^{-1}_{2}=\varepsilon$, i.e., $\sigma=\varepsilon$ or $(1\,3)$. 
	Hence 
	\begin{equation} \label{eq2}
	\theta^{-1} \prescript{\tau \sigma \tau^{-1}}{}N_v^\ast \theta \subseteq \prescript{\sigma}{}N_l^\circ \iff \sigma=\varepsilon \;\text{or}\; (1\,3).
	\end{equation}
	Also, for $ \sigma \in \{\varepsilon, (1\,3)\}$,
	we obtain $ \alpha_{\sigma^{-1}1}^{-1} \varphi_{\tau^{-1}1} \alpha_1 \in \prescript{\sigma}{1}N_l^\circ$ and $ \alpha_{\sigma^{-1}3}^{-1} \varphi_{\tau^{-1}3} \alpha_3 \in \prescript{\sigma}{3}N_l^\circ$. Therefore $ \alpha_{\sigma^{-1}1}^{-1} \prescript{\tau \sigma \tau^{-1}}{\tau^{-1}1}N_v^\ast \alpha_1 \subseteq \prescript{\sigma}{1}N_l^\circ$ and $ \alpha_{\sigma^{-1}3}^{-1} \prescript{\tau \sigma \tau^{-1}}{\tau^{-1}3}N_v^\ast \alpha_3\subseteq \prescript{\sigma}{3}N_l^\circ$.
	From (\ref{eq1}) and (\ref{eq2}) we obtain $\prescript{\sigma}{}N_l^\circ= \theta^{-1} \prescript{\tau \sigma \tau^{-1}}{}N_v^\ast \theta$ $\iff$ $\sigma=\varepsilon$ or $(1\,3)$. 
	Also for $ \sigma \in \{\varepsilon, (1\,3)\}$, we have $\prescript{\sigma}{1}N_l^\circ= \alpha_{\sigma^{-1}1}^{-1} \prescript{\tau \sigma \tau^{-1}}{\tau^{-1}1}N_v^\ast \alpha_1$ and $\prescript{\sigma}{3}N_l^\circ= \alpha_{\sigma^{-1}3}^{-1} \prescript{\tau \sigma \tau^{-1}}{\tau^{-1}3}N_v^\ast \alpha_3$.
	
	$(2)$ and $(3)$ can be proved on similar lines.
\end{proof}

For $\sigma\in\{(1\,3),(2\,3),(1\,2)\}$, Table \ref{table5} shows connections between components of $\sigma$-A-nuclei of a quasigroup $(Q,\ast)$ and components of $\sigma$-A-nuclei of its isostrophic images of the form $(Q,\circ)=(Q,\ast)(\tau, T)$, where $T=(\alpha,\beta,\gamma)$ and  $\alpha,\beta,\gamma$ are permutations of the set $Q$, for all $\tau \in S_3$. 

Note: We already have Table \ref{table2} that shows connections between components of $\varepsilon$-A-nuclei of a quasigroup and its isostrophic images. Hence we have omitted the case when $\sigma=\varepsilon$, as left (right,middle) $\varepsilon$-A-nucleus of a quasigroup coincides with the left (right,middle) A-nucleus of the quasigroup.

\begin{table}[h!]
	\caption{Connections between components of $\sigma$-A-nuclei of a quasigroup and its isostrophic images.} \label{table5}
	\centering
	\begin{adjustbox}{width=\textwidth,center}
		\begin{tabular}{|l|l|l|l|l|l|l|} \hline
			&\multicolumn{1}{c}{$(\varepsilon,T)$} & \multicolumn{1}{|c}{$((1\,2),T)$} & \multicolumn{1}{|c}{$((1\,3),T)$}  & \multicolumn{1}{|c}{$((2\,3),T)$}  & \multicolumn{1}{|c}{$((1\,2\,3),T)$}  & \multicolumn{1}{|c|}{$((1\,3\,2),T)$} \\\hline 	\hline
			
			\multicolumn{1}{|l}{$\prescript{(1\,3)}{1}N_l^\circ$}& 
			\multicolumn{1}{|l}{$\gamma^{-1} \prescript{(1\,3)}{1}N^{\ast}_{l} \alpha$}& \multicolumn{1}{|l}{$\gamma^{-1}\prescript{(2\,3)}{2}N_r^\ast \alpha$}& \multicolumn{1}{|l}{$\gamma^{-1}\prescript{(1\,3)}{3}N_l^\ast \alpha$}& \multicolumn{1}{|l}{$\gamma^{-1}\prescript{(1\,2)}{1}N_m^\ast \alpha$}& \multicolumn{1}{|l}{$\gamma^{-1}\prescript{(2\,3)}{3}N_r^\ast \alpha$}& \multicolumn{1}{|l|}{$\gamma^{-1}\prescript{(1\,2)}{2}N_m^\ast \alpha$} \\ \hline
			
			\multicolumn{1}{|l}{$\prescript{(1\,3)}{3}N_l^\circ$}& 
			\multicolumn{1}{|l}{$\alpha^{-1} \prescript{(1\,3)}{3}N^{\ast}_{l} \gamma$}& \multicolumn{1}{|l}{$\alpha^{-1}\prescript{(2\,3)}{3}N_r^\ast \gamma$}& \multicolumn{1}{|l}{$\alpha^{-1}\prescript{(1\,3)}{1}N_l^\ast \gamma$}& \multicolumn{1}{|l}{$\alpha^{-1}\prescript{(1\,2)}{2}N_m^\ast \gamma$}& \multicolumn{1}{|l}{$\alpha^{-1}\prescript{(2\,3)}{2}N_r^\ast \gamma$}& \multicolumn{1}{|l|}{$\alpha^{-1}\prescript{(1\,2)}{1}N_m^\ast \gamma$} \\ \hline
			
			\multicolumn{1}{|l}{$\prescript{(2\,3)}{2}N_r^\circ$}& 
			\multicolumn{1}{|l}{$\gamma^{-1} \prescript{(2\,3)}{2}N^{\ast}_{r} \beta$}& \multicolumn{1}{|l}{$\gamma^{-1}\prescript{(1\,3)}{1}N_l^\ast \beta$}& \multicolumn{1}{|l}{$\gamma^{-1}\prescript{(1\,2)}{2}N_m^\ast \beta$}& \multicolumn{1}{|l}{$\gamma^{-1}\prescript{(2\,3)}{3}N_r^\ast \beta$}& \multicolumn{1}{|l}{$\gamma^{-1}\prescript{(1\,2)}{1}N_m^\ast \beta$}& \multicolumn{1}{|l|}{$\gamma^{-1}\prescript{(1\,3)}{3}N_l^\ast \beta$} \\ \hline
			
			\multicolumn{1}{|l}{$\prescript{(2\,3)}{3}N_r^\circ$}& 
			\multicolumn{1}{|l}{$\beta^{-1} \prescript{(2\,3)}{3}N^{\ast}_{r} \gamma$}& \multicolumn{1}{|l}{$\beta^{-1}\prescript{(1\,3)}{3}N_l^\ast \gamma$}& \multicolumn{1}{|l}{$\beta^{-1}\prescript{(1\,2)}{1}N_m^\ast \gamma$}& \multicolumn{1}{|l}{$\beta^{-1}\prescript{(2\,3)}{2}N_r^\ast \gamma$}& \multicolumn{1}{|l}{$\beta^{-1}\prescript{(1\,2)}{2}N_m^\ast \gamma$}& \multicolumn{1}{|l|}{$\beta^{-1}\prescript{(1\,3)}{1}N_l^\ast \gamma$} \\ \hline
			
			\multicolumn{1}{|l}{$\prescript{(1\,2)}{1}N_m^\circ$}& 
			\multicolumn{1}{|l}{$\beta^{-1} \prescript{(1\,2)}{1}N^{\ast}_{m} \alpha$}& \multicolumn{1}{|l}{$\beta^{-1}\prescript{(1\,2)}{2}N_m^\ast \alpha$}& \multicolumn{1}{|l}{$\beta^{-1}\prescript{(2\,3)}{3}N_r^\ast \alpha$}& \multicolumn{1}{|l}{$\beta^{-1}\prescript{(1\,3)}{1}N_l^\ast \alpha$}& \multicolumn{1}{|l}{$\beta^{-1}\prescript{(1\,3)}{3}N_l^\ast \alpha$}& \multicolumn{1}{|l|}{$\beta^{-1}\prescript{(2\,3)}{2}N_r^\ast \alpha$} \\ \hline
			
			\multicolumn{1}{|l}{$\prescript{(1\,2)}{2}N_m^\circ$}& 
			\multicolumn{1}{|l}{$\alpha^{-1} \prescript{(1\,2)}{2}N^{\ast}_{m} \beta$}& \multicolumn{1}{|l}{$\alpha^{-1}\prescript{(1\,2)}{1}N_m^\ast \beta$}& \multicolumn{1}{|l}{$\alpha^{-1}\prescript{(2\,3)}{2}N_r^\ast \beta$}& \multicolumn{1}{|l}{$\alpha^{-1}\prescript{(1\,3)}{3}N_l^\ast \beta$}& \multicolumn{1}{|l}{$\alpha^{-1}\prescript{(1\,3)}{1}N_l^\ast \beta$}& \multicolumn{1}{|l|}{$\alpha^{-1}\prescript{(2\,3)}{3}N_r^\ast \beta$} \\ \hline
			
		\end{tabular}
	\end{adjustbox}
\end{table}

\begin{corollary}
	Isostrophic quasigroups have isomorphic components of $\sigma$-A-nucleus, for $\sigma \in \{\varepsilon,(1\,2),(1\,3),(2\,3)\}$.
\end{corollary}
\begin{proof}
	The proof follows from Table \ref{table5}.
\end{proof}


Let $(Q,\circ)$ be a $\tau$-parastrophe of a quasigroup $(Q,\ast)$, where $\tau \in S_3$. Then $(Q,\circ)=(Q,\ast^{\tau})= (Q,\ast)(\tau,\boldsymbol \varepsilon)$.
Here the symbol $\boldsymbol \varepsilon$ denotes $(\varepsilon,\varepsilon,\varepsilon)$,  where $\varepsilon$ is the identity mapping.

The following theorem gives a complete description of left, right and middle $\sigma$-A-nuclei of $\tau$-parastrophes $(Q,\circ)$ of a quasigroup $(Q,\ast)$ and their respective component sets, for any $\sigma, \tau \in S_3$.


\begin{theorem} \label{thm5}
	Let $(Q,\ast)$ be a quasigroup and $(Q,\circ)$ be a $\tau$-parastrophe of $(Q,\ast)$, i.e., $\circ=\ast^{\tau}$, where $\tau \in S_3$. If $\theta=(\tau,\boldsymbol \varepsilon )$, then for all $\sigma \in S_3$ the following hold:
	\begin{enumerate}
		\item $\prescript{\sigma}{}N_l^\circ= \theta^{-1} \prescript{\tau \sigma \tau^{-1}}{}N_v^\ast \theta$, where $v=
		\begin{cases*}
		r & if $\tau^{-1} 2=1$,\\ 
		l & if $\tau^{-1} 2=2$, \\
		m & if $\tau^{-1} 2=3.$
		\end{cases*}$
		Further, 
		$\prescript{\sigma}{1}N_l^\circ= \prescript{\tau \sigma \tau^{-1}}{\tau^{-1}1}N_v^\ast $ and $\prescript{\sigma}{3}N_l^\circ= \prescript{\tau \sigma \tau^{-1}}{\tau^{-1}3}N_v^\ast $. 
		
		\item $\prescript{\sigma}{}N_r^\circ= \theta^{-1} \prescript{\tau \sigma \tau^{-1}}{}N_v^\ast \theta$, where $v=
		\begin{cases*}
		r & if $\tau^{-1} 1=1$,\\ 
		l & if $\tau^{-1} 1=2$, \\
		m & if $\tau^{-1} 1=3.$
		\end{cases*}$
		Further,
		$\prescript{\sigma}{2}N_r^\circ= \prescript{\tau \sigma  \tau^{-1}}{\tau^{-1}2}N_v^\ast $ and $\prescript{\sigma}{3}N_r^\circ= \prescript{\tau \sigma  \tau^{-1}}{\tau^{-1}3}N_v^\ast $. 
		
		\item $\prescript{\sigma}{}N_m^\circ= \theta^{-1} \prescript{\tau \sigma \tau^{-1}}{}N_v^\ast \theta$, where $v=
		\begin{cases*}
		r & if $\tau^{-1} 3=1$,\\ 
		l & if $\tau^{-1} 3=2$, \\
		m & if $\tau^{-1} 3=3.$
		\end{cases*}$
		Further, 
		$\prescript{\sigma}{1}N_m^\circ= \prescript{\tau \sigma  \tau^{-1}}{\tau^{-1}1}N_v^\ast$ and $\prescript{\sigma}{2}N_m^\circ= \prescript{\tau \sigma  \tau^{-1}}{\tau^{-1}2}N_v^\ast$. 
	\end{enumerate}
\end{theorem}

\begin{proof}
	$(1)$ We shall show the inclusions in both directions. 
	Let $\varphi \in \prescript{\sigma}{}N_l^\circ$. 
	Then there exist permutations $\varphi_1,\varphi_2,\varphi_3$ of $Q$ such that $\varphi= (\sigma, (\varphi_1,\varphi_2,\varphi_3))=(\sigma,R)$ (say), where $\varphi_2=\varepsilon$. 
	Therefore $\varphi\in Aus(Q,\circ)$, and by Theorem \ref{thm1}, we have $\theta \varphi\theta^{-1} \in Aus(Q,\ast)$.
	Also 
	\begin{align*}
	\theta \varphi\theta^{-1}&= (\tau,\boldsymbol \varepsilon ) (\sigma, R) (\tau, \boldsymbol \varepsilon )^{-1}\\
	&= (\tau \sigma \tau ^{-1}, R^{\tau ^{-1}})\\
	&= (\tau \sigma \tau ^{-1}, (\varphi_{\tau 1} ,\varphi_{\tau 2} , \varphi_{\tau 3})) \\
	&=(\tau \sigma \tau ^{-1}, (\alpha_1,\alpha_2,\alpha_3))\text{ (say)}.
	\end{align*}
	Then on comparing the components, we get $\alpha_i= \varphi_{\tau i}$, i.e., $\alpha_{\tau^{-1}i}= \varphi_{i}$ for $i=1,2,3$. For $i=2$, $\alpha_{\tau^{-1}2}= \varphi_{2}= \varepsilon$, which implies $\theta \varphi\theta^{-1} \in \prescript{\tau \sigma \tau^{-1}}{}N_v^\ast $.
	Hence 
	$\theta \prescript{\sigma}{}N_l^\circ \theta^{-1} \subseteq \prescript{\tau \sigma \tau^{-1}}{}N_v^\ast.$
	Also we get $\varphi_{i}= \alpha_{\tau^{-1}i} \in  \prescript{\tau \sigma \tau^{-1}}{\tau^{-1}i}N_v^\ast $, for $i=1,2,3$. 
	Thus $\prescript{\sigma}{1}N_l^\circ\subseteq \prescript{\tau \sigma \tau^{-1}}{\tau^{-1}1}N_v^\ast $ and $\prescript{\sigma}{3}N_l^\circ\subseteq \prescript{\tau \sigma \tau^{-1}}{\tau^{-1}3}N_v^\ast $.
	
	
	Conversely, let $\varphi \in \prescript{\tau \sigma \tau^{-1}}{}N_v^\ast $.
	Then there exist permutations $\varphi_1,\varphi_2,\varphi_3$ of $Q$ such that $\varphi=(\tau \sigma \tau^{-1}, (\varphi_1, \varphi_2 , \varphi_3))=(\tau \sigma \tau^{-1},R)$ (say),  where $\varphi_{\tau^{-1} 2}=\varepsilon$.
	We have $\varphi\in Aus(Q,\ast)$ and by Theorem \ref{thm1} $\theta^{-1}\varphi\theta \in Aus(Q,\circ)$.
	Also 
	\[\theta^{-1}\varphi\theta = (\tau,\boldsymbol \varepsilon)^{-1} (\tau \sigma \tau^{-1},R) (\tau,\boldsymbol \varepsilon)= (\sigma, R^{\tau}).\]
	On substituting the value of $R$, and as $\varphi_{\tau^{-1} 2}=\varepsilon$, we get $\theta^{-1}\varphi\theta = (\sigma, (\varphi_{\tau^{-1}1} , \varepsilon, \varphi_{\tau^{-1}3})) \in \prescript{\sigma}{}N_l^\circ $.  
	Hence $\theta^{-1} \prescript{\tau \sigma \tau^{-1}}{}N_v^\ast \theta \subseteq \prescript{\sigma}{}N_l^\circ.$ 
	Also observe that
	$\varphi_{\tau^{-1}1} \in \prescript{\sigma}{1}N_l^\circ$ and $\varphi_{\tau^{-1}3} \in \prescript{\sigma}{3}N_l^\circ$, thus $ \prescript{\tau \sigma \tau^{-1}}{\tau^{-1}1}N_v^\ast \subseteq \prescript{\sigma}{1}N_l^\circ $ and $ \prescript{\tau \sigma \tau^{-1}}{\tau^{-1}3}N_v^\ast \subseteq \prescript{\sigma}{3}N_l^\circ$. 
	
	$(2)$ and $(3)$ can be proved on similar lines.
\end{proof}

Table \ref{table6} shows connections between components of $\sigma$-A-nuclei of a quasigroup $(Q,\ast)$ and components of $\sigma$-A-nuclei of its $\tau$-parastrophes $(Q,\circ)=(Q,\ast^{\tau})$, for all $\sigma, \tau \in S_3$.

Note that the $\tau$-parastrophes $(Q,\circ)=(Q,\ast^{\tau})=(Q,\ast)(\tau,\boldsymbol{\varepsilon})$ can also be considered as isostrophic images of quasigroup $(Q,\ast)$ with isostrophisms $(\tau,\boldsymbol{\varepsilon})$.

\begin{table}[h!]
	\caption{Connections between components of $\sigma$-A-nuclei of a quasigroup and its $\tau$-parastrophes.} \label{table6}
	\centering
	\begin{adjustbox}{width=\textwidth,center}
		\begin{tabular}{|l|l|l|l|l|l|l|l} \hline
			$\sigma$$\backslash \tau$ & &\multicolumn{1}{c}{$\varepsilon$} & \multicolumn{1}{|c}{$(1\,2)$} & \multicolumn{1}{|c}{$(1\,3)$}  & \multicolumn{1}{|c}{$(2\,3)$}  & \multicolumn{1}{|c}{$(1\,2\,3)$}  & \multicolumn{1}{|c|}{$(1\,3\,2)$} \\ \hline 	\hline
			
			\multirow{6}{*}{$\varepsilon$}  & 	\multicolumn{1}{c}{$\prescript{\varepsilon}{1}N_l^\circ$}& 
			\multicolumn{1}{|c}{$\prescript{\varepsilon}{1}N^{\ast}_{l} $}& \multicolumn{1}{|c}{$\prescript{\varepsilon}{2}N_r^\ast $}& \multicolumn{1}{|c}{$\prescript{\varepsilon}{3}N_l^\ast $}& \multicolumn{1}{|c}{$\prescript{\varepsilon}{1}N_m^\ast $}& \multicolumn{1}{|c}{$\prescript{\varepsilon}{3}N_r^\ast $}& \multicolumn{1}{|c|}{$\prescript{\varepsilon}{2}N_m^\ast $} \\ \cline{2-8}

			&\multicolumn{1}{c}{$\prescript{\varepsilon}{3}N_l^\circ$}& 
			\multicolumn{1}{|c}{$\prescript{\varepsilon}{3}N^{\ast}_{l} $}& \multicolumn{1}{|c}{$\prescript{\varepsilon}{3}N_r^\ast $}& \multicolumn{1}{|c}{$\prescript{\varepsilon}{1}N_l^\ast $}& \multicolumn{1}{|c}{$\prescript{\varepsilon}{2}N_m^\ast $}& \multicolumn{1}{|c}{$\prescript{\varepsilon}{2}N_r^\ast $}& \multicolumn{1}{|c|}{$\prescript{\varepsilon}{1}N_m^\ast $}\\ \cline{2-8}
			
			& 	\multicolumn{1}{c}{$\prescript{\varepsilon}{2}N_r^\circ$}& 
			\multicolumn{1}{|c}{$\prescript{\varepsilon}{2}N^{\ast}_{r} $}& \multicolumn{1}{|c}{$\prescript{\varepsilon}{1}N_l^\ast $}& \multicolumn{1}{|c}{$\prescript{\varepsilon}{2}N_m^\ast $}& \multicolumn{1}{|c}{$\prescript{\varepsilon}{3}N_r^\ast $}& \multicolumn{1}{|c}{$\prescript{\varepsilon}{1}N_m^\ast $}& \multicolumn{1}{|c|}{$\prescript{\varepsilon}{3}N_l^\ast $} \\ \cline{2-8}

			&\multicolumn{1}{c}{$\prescript{\varepsilon}{3}N_r^\circ$}& 
			\multicolumn{1}{|c}{$\prescript{\varepsilon}{3}N^{\ast}_{r} $}& \multicolumn{1}{|c}{$\prescript{\varepsilon}{3}N_l^\ast $}& \multicolumn{1}{|c}{$\prescript{\varepsilon}{1}N_m^\ast $}& \multicolumn{1}{|c}{$\prescript{\varepsilon}{2}N_r^\ast $}& \multicolumn{1}{|c}{$\prescript{\varepsilon}{2}N_m^\ast $}& \multicolumn{1}{|c|}{$\prescript{\varepsilon}{1}N_l^\ast $}\\ \cline{2-8}
			
			& 	\multicolumn{1}{c}{$\prescript{\varepsilon}{1}N_m^\circ$}& 
			\multicolumn{1}{|c}{$\prescript{\varepsilon}{1}N^{\ast}_{m} $}& \multicolumn{1}{|c}{$\prescript{\varepsilon}{2}N_m^\ast $}& \multicolumn{1}{|c}{$\prescript{\varepsilon}{3}N_r^\ast $}& \multicolumn{1}{|c}{$\prescript{\varepsilon}{1}N_l^\ast $}& \multicolumn{1}{|c}{$\prescript{\varepsilon}{3}N_l^\ast $}& \multicolumn{1}{|c|}{$\prescript{\varepsilon}{2}N_r^\ast $} \\ \cline{2-8}

			&\multicolumn{1}{c}{$\prescript{\varepsilon}{2}N_m^\circ$}& 
			\multicolumn{1}{|c}{$\prescript{\varepsilon}{2}N^{\ast}_{m} $}& \multicolumn{1}{|c}{$\prescript{\varepsilon}{1}N_m^\ast $}& \multicolumn{1}{|c}{$\prescript{\varepsilon}{2}N_r^\ast $}& \multicolumn{1}{|c}{$\prescript{\varepsilon}{3}N_l^\ast $}& \multicolumn{1}{|c}{$\prescript{\varepsilon}{1}N_l^\ast $}& \multicolumn{1}{|c|}{$\prescript{\varepsilon}{3}N_r^\ast $}\\ \hline 
			
			\multirow{6}{*}{$(1\,2)$}  & 	\multicolumn{1}{c}{$\prescript{(1\,2)}{1}N_l^\circ$}& 
			\multicolumn{1}{|c}{$\prescript{(1\,2)}{1}N^{\ast}_{l} $}& \multicolumn{1}{|c}{$\prescript{(1\,2)}{2}N_r^\ast $}& \multicolumn{1}{|c}{$\prescript{(2\,3)}{3}N_l^\ast $}& \multicolumn{1}{|c}{$\prescript{(1\,3)}{1}N_m^\ast $}& \multicolumn{1}{|c}{$\prescript{(1\,3)}{3}N_r^\ast $}& \multicolumn{1}{|c|}{$\prescript{(2\,3)}{2}N_m^\ast $} \\ \cline{2-8}

			&\multicolumn{1}{c}{$\prescript{(1\,2)}{3}N_l^\circ$}& 
			\multicolumn{1}{|c}{$\prescript{(1\,2)}{3}N^{\ast}_{l} $}& \multicolumn{1}{|c}{$\prescript{(1\,2)}{3}N_r^\ast $}& \multicolumn{1}{|c}{$\prescript{(2\,3)}{1}N_l^\ast $}& \multicolumn{1}{|c}{$\prescript{(1\,3)}{2}N_m^\ast $}& \multicolumn{1}{|c}{$\prescript{(1\,3)}{2}N_r^\ast $}& \multicolumn{1}{|c|}{$\prescript{(2\,3)}{1}N_m^\ast $}\\ \cline{2-8}
			
			& 	\multicolumn{1}{c}{$\prescript{(1\,2)}{2}N_r^\circ$}& 
			\multicolumn{1}{|c}{$\prescript{(1\,2)}{2}N^{\ast}_{r} $}& \multicolumn{1}{|c}{$\prescript{(1\,2)}{1}N_l^\ast $}& \multicolumn{1}{|c}{$\prescript{(2\,3)}{2}N_m^\ast $}& \multicolumn{1}{|c}{$\prescript{(1\,3)}{3}N_r^\ast $}& \multicolumn{1}{|c}{$\prescript{(1\,3)}{1}N_m^\ast $}& \multicolumn{1}{|c|}{$\prescript{(2\,3)}{3}N_l^\ast $} \\ \cline{2-8}

			&\multicolumn{1}{c}{$\prescript{(1\,2)}{3}N_r^\circ$}& 
			\multicolumn{1}{|c}{$\prescript{(1\,2)}{3}N^{\ast}_{r} $}& \multicolumn{1}{|c}{$\prescript{(1\,2)}{3}N_l^\ast $}& \multicolumn{1}{|c}{$\prescript{(2\,3)}{1}N_m^\ast $}& \multicolumn{1}{|c}{$\prescript{(1\,3)}{2}N_r^\ast $}& \multicolumn{1}{|c}{$\prescript{(1\,3)}{2}N_m^\ast $}& \multicolumn{1}{|c|}{$\prescript{(2\,3)}{1}N_l^\ast $}\\ \cline{2-8}
			
			& 	\multicolumn{1}{c}{$\prescript{(1\,2)}{1}N_m^\circ$}& 
			\multicolumn{1}{|c}{$\prescript{(1\,2)}{1}N^{\ast}_{m} $}& \multicolumn{1}{|c}{$\prescript{(1\,2)}{2}N_m^\ast $}& \multicolumn{1}{|c}{$\prescript{(2\,3)}{3}N_r^\ast $}& \multicolumn{1}{|c}{$\prescript{(1\,3)}{1}N_l^\ast $}& \multicolumn{1}{|c}{$\prescript{(1\,3)}{3}N_l^\ast $}& \multicolumn{1}{|c|}{$\prescript{(2\,3)}{2}N_r^\ast $} \\ \cline{2-8}

			&\multicolumn{1}{c}{$\prescript{(1\,2)}{2}N_m^\circ$}& 
			\multicolumn{1}{|c}{$\prescript{(1\,2)}{2}N^{\ast}_{m} $}& \multicolumn{1}{|c}{$\prescript{(1\,2)}{1}N_m^\ast $}& \multicolumn{1}{|c}{$\prescript{(2\,3)}{2}N_r^\ast $}& \multicolumn{1}{|c}{$\prescript{(1\,3)}{3}N_l^\ast $}& \multicolumn{1}{|c}{$\prescript{(1\,3)}{1}N_l^\ast $}& \multicolumn{1}{|c|}{$\prescript{(2\,3)}{3}N_r^\ast $}\\ \hline

			\multirow{6}{*}{$(1\,3)$}  & 	\multicolumn{1}{c}{$\prescript{(1\,3)}{1}N_l^\circ$}& 
			\multicolumn{1}{|c}{$\prescript{(1\,3)}{1}N^{\ast}_{l} $}& \multicolumn{1}{|c}{$\prescript{(2\,3)}{2}N_r^\ast $}& \multicolumn{1}{|c}{$\prescript{(1\,3)}{3}N_l^\ast $}& \multicolumn{1}{|c}{$\prescript{(1\,2)}{1}N_m^\ast $}& \multicolumn{1}{|c}{$\prescript{(2\,3)}{3}N_r^\ast $}& \multicolumn{1}{|c|}{$\prescript{(1\,2)}{2}N_m^\ast $} \\ \cline{2-8}

			&\multicolumn{1}{c}{$\prescript{(1\,3)}{3}N_l^\circ$}& 
			\multicolumn{1}{|c}{$\prescript{(1\,3)}{3}N^{\ast}_{l} $}& \multicolumn{1}{|c}{$\prescript{(2\,3)}{3}N_r^\ast $}& \multicolumn{1}{|c}{$\prescript{(1\,3)}{1}N_l^\ast $}& \multicolumn{1}{|c}{$\prescript{(1\,2)}{2}N_m^\ast $}& \multicolumn{1}{|c}{$\prescript{(2\,3)}{2}N_r^\ast $}& \multicolumn{1}{|c|}{$\prescript{(1\,2)}{1}N_m^\ast $}\\ \cline{2-8}
			
			& \multicolumn{1}{c}{$\prescript{(1\,3)}{2}N_r^\circ$}& 
			\multicolumn{1}{|c}{$\prescript{(1\,3)}{2}N^{\ast}_{r} $}& \multicolumn{1}{|c}{$\prescript{(2\,3)}{1}N_l^\ast $}& \multicolumn{1}{|c}{$\prescript{(1\,3)}{2}N_m^\ast $}& \multicolumn{1}{|c}{$\prescript{(1\,2)}{3}N_r^\ast $}& \multicolumn{1}{|c}{$\prescript{(2\,3)}{1}N_m^\ast $}& \multicolumn{1}{|c|}{$\prescript{(1\,2)}{3}N_l^\ast $} \\ \cline{2-8}

			&\multicolumn{1}{c}{$\prescript{(1\,3)}{3}N_r^\circ$}& 
			\multicolumn{1}{|c}{$\prescript{(1\,3)}{3}N^{\ast}_{r} $}& \multicolumn{1}{|c}{$\prescript{(2\,3)}{3}N_l^\ast $}& \multicolumn{1}{|c}{$\prescript{(1\,3)}{1}N_m^\ast $}& \multicolumn{1}{|c}{$\prescript{(1\,2)}{2}N_r^\ast $}& \multicolumn{1}{|c}{$\prescript{(2\,3)}{2}N_m^\ast $}& \multicolumn{1}{|c|}{$\prescript{(1\,2)}{1}N_l^\ast $}\\ \cline{2-8}
			
			& \multicolumn{1}{c}{$\prescript{(1\,3)}{1}N_m^\circ$}& 
			\multicolumn{1}{|c}{$\prescript{(1\,3)}{1}N^{\ast}_{m} $}& \multicolumn{1}{|c}{$\prescript{(2\,3)}{2}N_m^\ast $}& \multicolumn{1}{|c}{$\prescript{(1\,3)}{3}N_r^\ast $}& \multicolumn{1}{|c}{$\prescript{(1\,2)}{1}N_l^\ast $}& \multicolumn{1}{|c}{$\prescript{(2\,3)}{3}N_l^\ast $}& \multicolumn{1}{|c|}{$\prescript{(1\,2)}{2}N_r^\ast $} \\ \cline{2-8}

			&\multicolumn{1}{c}{$\prescript{(1\,3)}{2}N_m^\circ$}& 
			\multicolumn{1}{|c}{$\prescript{(1\,3)}{2}N^{\ast}_{m} $}& \multicolumn{1}{|c}{$\prescript{(2\,3)}{1}N_m^\ast $}& \multicolumn{1}{|c}{$\prescript{(1\,3)}{2}N_r^\ast $}& \multicolumn{1}{|c}{$\prescript{(1\,2)}{3}N_l^\ast $}& \multicolumn{1}{|c}{$\prescript{(2\,3)}{1}N_l^\ast $}& \multicolumn{1}{|c|}{$\prescript{(1\,2)}{3}N_r^\ast $}\\ \hline
			
			\multirow{6}{*}{$(2\,3)$}  & 	\multicolumn{1}{c}{$\prescript{(2\,3)}{1}N_l^\circ$}& 
			\multicolumn{1}{|c}{$\prescript{(2\,3)}{1}N^{\ast}_{l} $}& \multicolumn{1}{|c}{$\prescript{(1\,3)}{2}N_r^\ast $}& \multicolumn{1}{|c}{$\prescript{(1\,2)}{3}N_l^\ast $}& \multicolumn{1}{|c}{$\prescript{(2\,3)}{1}N_m^\ast $}& \multicolumn{1}{|c}{$\prescript{(1\,2)}{3}N_r^\ast $}& \multicolumn{1}{|c|}{$\prescript{(1\,3)}{2}N_m^\ast $} \\ \cline{2-8}

			&\multicolumn{1}{c}{$\prescript{(2\,3)}{3}N_l^\circ$}& 
			\multicolumn{1}{|c}{$\prescript{(2\,3)}{3}N^{\ast}_{l} $}& \multicolumn{1}{|c}{$\prescript{(1\,3)}{3}N_r^\ast $}& \multicolumn{1}{|c}{$\prescript{(1\,2)}{1}N_l^\ast $}& \multicolumn{1}{|c}{$\prescript{(2\,3)}{2}N_m^\ast $}& \multicolumn{1}{|c}{$\prescript{(1\,2)}{2}N_r^\ast $}& \multicolumn{1}{|c|}{$\prescript{(1\,3)}{1}N_m^\ast $}\\ \cline{2-8}
			
			& 	\multicolumn{1}{c}{$\prescript{(2\,3)}{2}N_r^\circ$}& 
			\multicolumn{1}{|c}{$\prescript{(2\,3)}{2}N^{\ast}_{r} $}& \multicolumn{1}{|c}{$\prescript{(1\,3)}{1}N_l^\ast $}& \multicolumn{1}{|c}{$\prescript{(1\,2)}{2}N_m^\ast $}& \multicolumn{1}{|c}{$\prescript{(2\,3)}{3}N_r^\ast $}& \multicolumn{1}{|c}{$\prescript{(1\,2)}{1}N_m^\ast $}& \multicolumn{1}{|c|}{$\prescript{(1\,3)}{3}N_l^\ast $} \\ \cline{2-8}

			&\multicolumn{1}{c}{$\prescript{(2\,3)}{3}N_r^\circ$}& 
			\multicolumn{1}{|c}{$\prescript{(2\,3)}{3}N^{\ast}_{r} $}& \multicolumn{1}{|c}{$\prescript{(1\,3)}{3}N_l^\ast $}& \multicolumn{1}{|c}{$\prescript{(1\,2)}{1}N_m^\ast $}& \multicolumn{1}{|c}{$\prescript{(2\,3)}{2}N_r^\ast $}& \multicolumn{1}{|c}{$\prescript{(1\,2)}{2}N_m^\ast $}& \multicolumn{1}{|c|}{$\prescript{(1\,3)}{1}N_l^\ast $}\\ \cline{2-8}
			
			& 	\multicolumn{1}{c}{$\prescript{(2\,3)}{1}N_m^\circ$}& 
			\multicolumn{1}{|c}{$\prescript{(2\,3)}{1}N^{\ast}_{m} $}& \multicolumn{1}{|c}{$\prescript{(1\,3)}{2}N_m^\ast $}& \multicolumn{1}{|c}{$\prescript{(1\,2)}{3}N_r^\ast $}& \multicolumn{1}{|c}{$\prescript{(2\,3)}{1}N_l^\ast $}& \multicolumn{1}{|c}{$\prescript{(1\,2)}{3}N_l^\ast $}& \multicolumn{1}{|c|}{$\prescript{(1\,3)}{2}N_r^\ast $} \\ \cline{2-8}

			&\multicolumn{1}{c}{$\prescript{(2\,3)}{2}N_m^\circ$}& 
			\multicolumn{1}{|c}{$\prescript{(2\,3)}{2}N^{\ast}_{m} $}& \multicolumn{1}{|c}{$\prescript{(1\,3)}{1}N_m^\ast $}& \multicolumn{1}{|c}{$\prescript{(1\,2)}{2}N_r^\ast $}& \multicolumn{1}{|c}{$\prescript{(2\,3)}{3}N_l^\ast $}& \multicolumn{1}{|c}{$\prescript{(1\,2)}{1}N_l^\ast $}& \multicolumn{1}{|c|}{$\prescript{(1\,3)}{3}N_r^\ast $}\\ \hline
			
			\multirow{6}{*}{$(1\,2\,3)$}  & 	\multicolumn{1}{c}{$\prescript{(1\,2\,3)}{1}N_l^\circ$}& 
			\multicolumn{1}{|c}{$\prescript{(1\,2\,3)}{1}N^{\ast}_{l} $}& \multicolumn{1}{|c}{$\prescript{(1\,3\,2)}{2}N_r^\ast $}& \multicolumn{1}{|c}{$\prescript{(1\,3\,2)}{3}N_l^\ast $}& \multicolumn{1}{|c}{$\prescript{(1\,3\,2)}{1}N_m^\ast $}& \multicolumn{1}{|c}{$\prescript{(1\,2\,3)}{3}N_r^\ast $}& \multicolumn{1}{|c|}{$\prescript{(1\,2\,3)}{2}N_m^\ast $} \\ \cline{2-8}
			
			&\multicolumn{1}{c}{$\prescript{(1\,2\,3)}{3}N_l^\circ$}& 
			\multicolumn{1}{|c}{$\prescript{(1\,2\,3)}{3}N^{\ast}_{l} $}& \multicolumn{1}{|c}{$\prescript{(1\,3\,2)}{3}N_r^\ast $}& \multicolumn{1}{|c}{$\prescript{(1\,3\,2)}{1}N_l^\ast $}& \multicolumn{1}{|c}{$\prescript{(1\,3\,2)}{2}N_m^\ast $}& \multicolumn{1}{|c}{$\prescript{(1\,2\,3)}{2}N_r^\ast $}& \multicolumn{1}{|c|}{$\prescript{(1\,2\,3)}{1}N_m^\ast $}\\ \cline{2-8}
			
			& 	\multicolumn{1}{c}{$\prescript{(1\,2\,3)}{2}N_r^\circ$}& 
			\multicolumn{1}{|c}{$\prescript{(1\,2\,3)}{2}N^{\ast}_{r} $}& \multicolumn{1}{|c}{$\prescript{(1\,3\,2)}{1}N_l^\ast $}& \multicolumn{1}{|c}{$\prescript{(1\,3\,2)}{2}N_m^\ast $}& \multicolumn{1}{|c}{$\prescript{(1\,3\,2)}{3}N_r^\ast $}& \multicolumn{1}{|c}{$\prescript{(1\,2\,3)}{1}N_m^\ast $}& \multicolumn{1}{|c|}{$\prescript{(1\,2\,3)}{3}N_l^\ast $} \\ \cline{2-8}
			
			&\multicolumn{1}{c}{$\prescript{(1\,2\,3)}{3}N_r^\circ$}& 
			\multicolumn{1}{|c}{$\prescript{(1\,2\,3)}{3}N^{\ast}_{r} $}& \multicolumn{1}{|c}{$\prescript{(1\,3\,2)}{3}N_l^\ast $}& \multicolumn{1}{|c}{$\prescript{(1\,3\,2)}{1}N_m^\ast $}& \multicolumn{1}{|c}{$\prescript{(1\,3\,2)}{2}N_r^\ast $}& \multicolumn{1}{|c}{$\prescript{(1\,2\,3)}{2}N_m^\ast $}& \multicolumn{1}{|c|}{$\prescript{(1\,2\,3)}{1}N_l^\ast $}\\ \cline{2-8}
			
			& 	\multicolumn{1}{c}{$\prescript{(1\,2\,3)}{1}N_m^\circ$}& 
			\multicolumn{1}{|c}{$\prescript{(1\,2\,3)}{1}N^{\ast}_{m} $}& \multicolumn{1}{|c}{$\prescript{(1\,3\,2)}{2}N_m^\ast $}& \multicolumn{1}{|c}{$\prescript{(1\,3\,2)}{3}N_r^\ast $}& \multicolumn{1}{|c}{$\prescript{(1\,3\,2)}{1}N_l^\ast $}& \multicolumn{1}{|c}{$\prescript{(1\,2\,3)}{3}N_l^\ast $}& \multicolumn{1}{|c|}{$\prescript{(1\,2\,3)}{2}N_r^\ast $} \\ \cline{2-8}

			&\multicolumn{1}{c}{$\prescript{(1\,2\,3)}{2}N_m^\circ$}& 
			\multicolumn{1}{|c}{$\prescript{(1\,2\,3)}{2}N^{\ast}_{m} $}& \multicolumn{1}{|c}{$\prescript{(1\,3\,2)}{1}N_m^\ast $}& \multicolumn{1}{|c}{$\prescript{(1\,3\,2)}{2}N_r^\ast $}& \multicolumn{1}{|c}{$\prescript{(1\,3\,2)}{3}N_l^\ast $}& \multicolumn{1}{|c}{$\prescript{(1\,2\,3)}{1}N_l^\ast $}& \multicolumn{1}{|c|}{$\prescript{(1\,2\,3)}{3}N_r^\ast $}\\ \hline

			\multirow{6}{*}{$(1\,3\,2)$}  & 	\multicolumn{1}{c}{$\prescript{(1\,3\,2)}{1}N_l^\circ$}& 
			\multicolumn{1}{|c}{$\prescript{(1\,3\,2)}{1}N^{\ast}_{l} $}& \multicolumn{1}{|c}{$\prescript{(1\,2\,3)}{2}N_r^\ast $}& \multicolumn{1}{|c}{$\prescript{(1\,2\,3)}{3}N_l^\ast $}& \multicolumn{1}{|c}{$\prescript{(1\,2\,3)}{1}N_m^\ast $}& \multicolumn{1}{|c}{$\prescript{(1\,3\,2)}{3}N_r^\ast $}& \multicolumn{1}{|c|}{$\prescript{(1\,3\,2)}{2}N_m^\ast $} \\ \cline{2-8}
			
			&\multicolumn{1}{c}{$\prescript{(1\,3\,2)}{3}N_l^\circ$}& 
			\multicolumn{1}{|c}{$\prescript{(1\,3\,2)}{3}N^{\ast}_{l} $}& \multicolumn{1}{|c}{$\prescript{(1\,2\,3)}{3}N_r^\ast $}& \multicolumn{1}{|c}{$\prescript{(1\,2\,3)}{1}N_l^\ast $}& \multicolumn{1}{|c}{$\prescript{(1\,2\,3)}{2}N_m^\ast $}& \multicolumn{1}{|c}{$\prescript{(1\,3\,2)}{2}N_r^\ast $}& \multicolumn{1}{|c|}{$\prescript{(1\,3\,2)}{1}N_m^\ast $}\\ \cline{2-8}
			
			& 	\multicolumn{1}{c}{$\prescript{(1\,3\,2)}{2}N_r^\circ$}& 
			\multicolumn{1}{|c}{$\prescript{(1\,3\,2)}{2}N^{\ast}_{r} $}& \multicolumn{1}{|c}{$\prescript{(1\,2\,3)}{1}N_l^\ast $}& \multicolumn{1}{|c}{$\prescript{(1\,2\,3)}{2}N_m^\ast $}& \multicolumn{1}{|c}{$\prescript{(1\,2\,3)}{3}N_r^\ast $}& \multicolumn{1}{|c}{$\prescript{(1\,3\,2)}{1}N_m^\ast $}& \multicolumn{1}{|c|}{$\prescript{(1\,3\,2)}{3}N_l^\ast $} \\ \cline{2-8}

			&\multicolumn{1}{c}{$\prescript{(1\,3\,2)}{3}N_r^\circ$}& 
			\multicolumn{1}{|c}{$\prescript{(1\,3\,2)}{3}N^{\ast}_{r} $}& \multicolumn{1}{|c}{$\prescript{(1\,2\,3)}{3}N_l^\ast $}& \multicolumn{1}{|c}{$\prescript{(1\,2\,3)}{1}N_m^\ast $}& \multicolumn{1}{|c}{$\prescript{(1\,2\,3)}{2}N_r^\ast $}& \multicolumn{1}{|c}{$\prescript{(1\,3\,2)}{2}N_m^\ast $}& \multicolumn{1}{|c|}{$\prescript{(1\,3\,2)}{1}N_l^\ast $}\\ \cline{2-8}
			
			& 	\multicolumn{1}{c}{$\prescript{(1\,3\,2)}{1}N_m^\circ$}& 
			\multicolumn{1}{|c}{$\prescript{(1\,3\,2)}{1}N^{\ast}_{m} $}& \multicolumn{1}{|c}{$\prescript{(1\,2\,3)}{2}N_m^\ast $}& \multicolumn{1}{|c}{$\prescript{(1\,2\,3)}{3}N_r^\ast $}& \multicolumn{1}{|c}{$\prescript{(1\,2\,3)}{1}N_l^\ast $}& \multicolumn{1}{|c}{$\prescript{(1\,3\,2)}{3}N_l^\ast $}& \multicolumn{1}{|c|}{$\prescript{(1\,3\,2)}{2}N_r^\ast $} \\ \cline{2-8}

			&\multicolumn{1}{c}{$\prescript{(1\,3\,2)}{2}N_m^\circ$}& 
			\multicolumn{1}{|c}{$\prescript{(1\,3\,2)}{2}N^{\ast}_{m} $}& \multicolumn{1}{|c}{$\prescript{(1\,2\,3)}{1}N_m^\ast $}& \multicolumn{1}{|c}{$\prescript{(1\,2\,3)}{2}N_r^\ast $}& \multicolumn{1}{|c}{$\prescript{(1\,2\,3)}{3}N_l^\ast $}& \multicolumn{1}{|c}{$\prescript{(1\,3\,2)}{1}N_l^\ast $}& \multicolumn{1}{|c|}{$\prescript{(1\,3\,2)}{3}N_r^\ast $}\\ \hline 
			
		\end{tabular}
	\end{adjustbox}
\end{table}

\section{$\sigma$-A-nuclei in inverse quasigroups}
In this section we will study  properties of the $\sigma$-A-nuclei of various inverse quasigroups using the connections derived in Section \ref{sec3}. 
In Subsection \ref{subsec1}, we shall discuss on $(\alpha,\beta,\gamma)$-inverse quasigroups and in Subsection \ref{subsec2}, we shall discuss on $\lambda,\rho$ and $\mu$-inverse quasigroups.

\subsection{$\boldsymbol{(\alpha,\beta,\gamma)}$-inverse quasigroups} \label{subsec1}~\\
Let $(Q,\cdot)$ be a quasigroup . Then $(Q,\cdot)$ is an $(\alpha,\beta,\gamma)$-\textit{inverse quasigroup} if there exist permutations $\alpha,\beta,\gamma$ of the set $Q$ such that 
\begin{equation} \label{inv 1}
\alpha(x\cdot y) \cdot \beta x=\gamma y
\end{equation}  
for all $x,y\in Q$ \cite{Keedwell&Anthony, Shcherbacov&2017}.
From definition of autostrophism, (\ref{inv 1}) is true iff $\theta=((123),(\alpha,\beta,\gamma))$ is an autostrophism of $(Q,\cdot)$.
Thus from Theorem \ref{nucleus_iso} we get in an $(\alpha,\beta,\gamma)$-inverse quasigroup $\prescript{(1\,3)}{}N_l^A= \theta^{-1} \prescript{(2\,3)}{}N_r^A \theta$, $\prescript{(2\,3)}{}N_r^A= \theta^{-1} \prescript{(1\,2)}{}N_m^A \theta$ and $\prescript{(1\,2)}{}N_m^A= \theta^{-1} \prescript{(1\,3)}{}N_l^A \theta$, which implies that the left $(1\,3)$-A-nucleus is isomorphic to the right $(2\,3)$-A-nucleus, the right $(2\,3)$-A-nucleus is isomorphic to the middle $(1\,2)$-A-nucleus and the middle $(1\,2)$-A-nucleus is isomorphic to the left $(1\,3)$-A-nucleus. 
Hence $\prescript{(1\,3)}{}N_l^A$, $\prescript{(2\,3)}{}N_r^A$ and $\prescript{(1\,2)}{}N_m^A$ are isomorphic in an $(\alpha,\beta,\gamma)$-inverse quasigroup $(Q,\cdot)$.

Also observe from above that \[\prescript{(1\,3)}{}N_l^A= \theta^{-1} \prescript{(2\,3)}{}N_r^A \theta= (\theta^2)^{-1}\prescript{(1\,2)}{}N_m^A \theta^2= (\theta^3)^{-1} \prescript{(1\,3)}{}N_l^A \theta^3,\] which implies $\theta^3\in N(\prescript{(1\,3)}{}N_l^A).$
Similarly we have $\theta^3\in N(\prescript{(2\,3)}{}N_r^A)$ and  $\theta^3\in N(\prescript{(1\,2)}{}N_m^A)$. It can easily be seen that $\theta^3= (\beta\gamma\alpha,\gamma\alpha\beta,\alpha\beta\gamma)$. 
Thus in an $(\alpha,\beta,\gamma)$-inverse quasigroup, we have $(\beta\gamma\alpha,\gamma\alpha\beta,\alpha\beta\gamma) \in  N(\prescript{(1\,3)}{}N_l^A) \cap N(\prescript{(2\,3)}{}N_r^A) \cap N(\prescript{(1\,2)}{}N_m^A)$.

A quasigroup $(Q,\cdot)$ has the \textit{weak-inverse-property (WIP)} if there exists a permutation $J$ of the set $Q$ such that
\begin{equation} \label{inv 2}
x \cdot J (y \cdot x)= J y
\end{equation}
for all $x,y\in Q$ \cite{Belousov&67,Keedwell&Shcherbacov&2003}.
A WIP quasigroup is a $(J,\varepsilon,J)$-inverse quasigroup \cite{Belousov&67}, which implies that 
if $(Q,\cdot)$ is a WIP-quasigroup with respect to the permutation $J$, then $(J^2,J^2,J^2)  \in N(\prescript{(1\,3)}{}N_l^A) \cap N(\prescript{(2\,3)}{}N_r^A) \cap N(\prescript{(1\,2)}{}N_m^A)$.

A quasigroup $(Q,\cdot)$ has the \textit{crossed-inverse-property (CI)} if there exists a permutation $J$ of the set $Q$ such that
\begin{equation} \label{inv 3}
(x \cdot y) \cdot J x= y
\end{equation}
for all $x,y\in Q$ \cite{Keedwell&1999, Keedwell&Shcherbacov&2003}.
Thus a CI quasigroup is an $(\varepsilon,J,\varepsilon)$-inverse quasigroup, 
which implies that if $(Q,\cdot)$ is a CI-quasigroup with respect to the permutation $J$, then $ (J,J,J)\in   N(\prescript{(1\,3)}{}N_l^A) \cap N(\prescript{(2\,3)}{}N_r^A) \cap N(\prescript{(1\,2)}{}N_m^A)$.

A quasigroup $(Q,\cdot)$ has the $(r,s,t)$-\textit{inverse-property} if there exists a permutation $J$ of the set $Q$ such that
\begin{equation} \label{inv 4}
J^r (x \cdot y) \cdot J^s x= J^t y
\end{equation}
for all $x,y\in Q$ \cite{Keedwell&Shcherbacov&2004, Keedwell&Anthony}.
Thus an $(r,s,t)$-inverse quasigroup is a $(J^r,J^s,J^t)$-inverse quasigroup, which implies if $(Q,\cdot)$ is an $(r,s,t)$-quasigroup with respect to the permutation $J$ then $(J^{r+s+t},J^{r+s+t},J^{r+s+t}) \in N(\prescript{(1\,3)}{}N_l^A) \cap N(\prescript{(2\,3)}{}N_r^A) \cap N(\prescript{(1\,2)}{}N_m^A)$.

A quasigroup $(Q,\cdot)$ has the $m$-\textit{inverse-property} if there exists a permutation $J$ of the set $Q$ such that
\begin{equation} \label{inv 5}
J^m (x \cdot y) \cdot J^{m+1} x= J^m y
\end{equation}
for all $x,y\in Q$ \cite{Keedwell&2002, Keedwell&Shcherbacov&2003, Keedwell&Anthony}.
Thus an $m$-inverse-quasigroup is a $(J^m,J^{m+1},J^m)$-inverse quasigroup, which implies that if $(Q,\cdot)$ is an $m$-inverse-quasigroup with respect to the permutation $J$ then $(J^{3m+1},J^{3m+1},J^{3m+1}) \in  N(\prescript{(1\,3)}{}N_l^A) \cap N(\prescript{(2\,3)}{}N_r^A) \cap N(\prescript{(1\,2)}{}N_m^A)$.

Since an $(\alpha,\beta,\gamma)$-inverse quasigroup has $((1\,2\,3),(\alpha,\beta,\gamma))$ autostrophism, from Table \ref{table5} we conclude that in an $(\alpha,\beta,\gamma)$-inverse quasigroup, we have $\prescript{(1\,3)}{1}N_l^A= \gamma^{-1} \prescript{(2\,3)}{3}N_r^A \alpha$, $\prescript{(1\,3)}{3}N_l^A= \alpha^{-1} \prescript{(2\,3)}{2}N_r^A \gamma$, $\prescript{(2\,3)}{2}N_r^A= \gamma^{-1} \prescript{(1\,2)}{1}N_m^A \beta$, $\prescript{(2\,3)}{3}N_r^A= \beta^{-1} \prescript{(1\,2)}{2}N_m^A \gamma$, $\prescript{(1\,2)}{1}N_m^A= \beta^{-1} \prescript{(1\,3)}{3}N_l^A \alpha$, and  $\prescript{(1\,2)}{2}N_m^A= \alpha^{-1} \prescript{(1\,3)}{1}N_l^A \beta$. 
Therefore, in an $(\varepsilon,\beta,\gamma)$-inverse quasigroup, we have  $\prescript{(1\,3)}{1}N_l^A= \gamma^{-1} \prescript{(2\,3)}{3}N_r^A = \gamma^{-1} \beta^{-1} \prescript{(1\,2)}{2}N_m^A \gamma=  \gamma^{-1} \beta^{-1}  \prescript{(1\,3)}{1}N_l^A \beta \gamma $, which implies  $\beta\gamma \in N(\prescript{(1\,3)}{1}N_l^A)$.
Also, we have $\prescript{(1\,3)}{3}N_l^A=  \prescript{(2\,3)}{2}N_r^A \gamma =  \gamma^{-1} \prescript{(1\,2)}{1}N_m^A \beta \gamma=  \gamma^{-1} \beta^{-1} \prescript{(1\,3)}{3}N_l^A \beta \gamma$, and hence $\beta\gamma \in N(\prescript{(1\,3)}{3}N_l^A)$. 
Thus $\beta\gamma \in N(\prescript{(1\,3)}{1}N_l^A) \cap N(\prescript{(1\,3)}{3}N_l^A)$.

Similarly, in an $(\alpha,\varepsilon,\gamma)$-inverse quasigroup, we have $\gamma\alpha \in N(\prescript{(1\,2)}{1}N_m^A) \cap N(\prescript{(1\,2)}{2}N_m^A)$, and in an $(\alpha,\beta,\varepsilon)$-inverse quasigroup, we have $ \alpha\beta \in N(\prescript{(2\,3)}{2}N_r^A) \cap N(\prescript{(2\,3)}{3}N_r^A)$.

It can be observed that if $(Q,\cdot)$ is a WIP-quasigroup with respect to the permutation $J$, then $J^2 \in N(\prescript{(1\,2)}{1}N_m^A) \cap N(\prescript{(1\,2)}{2}N_m^A)$ and if $(Q,\cdot)$ is a CI-quasigroup with respect to the permutation $J$, then $ J\in  N(\prescript{(1\,3)}{1}N_l^A) \cap N(\prescript{(1\,3)}{3}N_l^A)$ and $ J \in N(\prescript{(2\,3)}{2}N_r^A) \cap N(\prescript{(2\,3)}{3}N_r^A)$.

\subsection{$\boldsymbol{\lambda,\rho}$ and $\boldsymbol{\mu}$-inverse quasigroups} \label{subsec2}~\\
A quasigroup $(Q,\cdot)$ is a $\lambda$-\textit{inverse quasigroup} if there exist permutations
$\lambda_1,\lambda_2,\lambda_3$ of the set $Q$ such that 
\begin{equation} \label{inv 6}
\lambda_1 x \cdot \lambda_2(x \cdot y)= \lambda_3 y
\end{equation} 
for all $ x,y\in Q$ \cite{Belousov&87, Keedwell&Anthony}. From definition of autostrophism, (\ref{inv 6}) holds if and only if $\theta= ((2\,3),(\lambda_1,\lambda_2,\lambda_3))$ is an autostrophism of $(Q,\cdot)$. 

From Theorem \ref{nucleus_iso}, we conclude that in a $\lambda$-inverse quasigroup, we have $\prescript{(1\,3)}{}N_l^A= \theta^{-1} \prescript{(1\,2)}{}N_m^A \theta$, $\prescript{(2\,3)}{}N_r^A= \theta^{-1} \prescript{(2\,3)}{}N_r^A \theta$ and $\prescript{(1\,2)}{}N_m^A= \theta^{-1} \prescript{(1\,3)}{}N_l^A \theta$, which implies 
left $(1\,3)$-A-nucleus is isomorphic to middle $(1\,2)$-A-nucleus.
Also, $\prescript{(1\,3)}{}N_l^A= \theta^{-1} \prescript{(1\,2)}{}N_m^A \theta= (\theta^2)^{-1}\prescript{(1\,3)}{}N_l^A \theta^2$, which gives $\theta^2\in N(\prescript{(1\,3)}{}N_l^A)$, $\prescript{(2\,3)}{}N_r^A= \theta^{-1} \prescript{(2\,3)}{}N_r^A \theta$ gives $\theta\in N(\prescript{(2\,3)}{}N_r^A)$ and
$\prescript{(1\,2)}{}N_m^A= \theta^{-1} \prescript{(1\,3)}{}N_l^A \theta= (\theta^2)^{-1}\prescript{(1\,2)}{}N_m^A \theta^2$ implies $\theta^2 \in N(\prescript{(1\,2)}{}N_m^A)$. 
Thus in a $\lambda$-inverse quasigroup, $\theta\in N(\prescript{(2\,3)}{}N_r^A)$ and $\theta^2 \in N(\prescript{(1\,3)}{}N_l^A) \cap N(\prescript{(1\,2)}{}N_m^A)$.

From Table \ref{table5}, we have in a $\lambda$-inverse quasigroup $\prescript{(1\,3)}{1}N_l^A= \lambda_3^{-1} \prescript{(1\,2)}{1}N_m^A \lambda_1$, $\prescript{(1\,3)}{3}N_l^A= \lambda_1^{-1} \prescript{(1\,2)}{2}N_m^A \lambda_3$, $\prescript{(2\,3)}{2}N_r^A= \lambda_3^{-1} \prescript{(2\,3)}{3}N_r^A \lambda_2$, $\prescript{(2\,3)}{3}N_r^A= \lambda_2^{-1} \prescript{(2\,3)}{2}N_r^A \lambda_3$, $\prescript{(1\,2)}{1}N_m^A= \lambda_2^{-1} \prescript{(1\,3)}{1}N_l^A \lambda_1$, and $\prescript{(1\,2)}{2}N_m^A= \lambda_1^{-1} \prescript{(1\,3)}{3}N_l^A \lambda_2$.

A quasigroup $(Q,\cdot)$ has the \textit{left-inverse-property (LIP)} with respect to a permutation $\lambda$ of $Q$ if 
\begin{equation} \label{inv 7}
\lambda x \cdot x  y= y
\end{equation}
for all $x,y \in Q$ \cite{Belousov&67, Shcherbacov&2017}. Note that (\ref{inv 7}) is true if and only if $((2\,3),(\lambda,\varepsilon,\varepsilon))$ is an autostrophism of $(Q,\cdot)$.  Therefore, an LIP-quasigroup is a $\lambda$-inverse quasigroup with $\lambda_1=\lambda$ and  $\lambda_2=\lambda_3=\varepsilon$. 
Also from (\ref{inv 7}) we have $\lambda^2 x(\lambda x \cdot xy)=x y$, i.e., $\lambda^2 x \cdot y = x \cdot y$, which gives $\lambda^2=\varepsilon$. Hence in LIP-quasigroup $\prescript{(1\,3)}{1}N_l^A= \prescript{(1\,2)}{1}N_m^A \lambda$, $\prescript{(1\,3)}{3}N_l^A= \lambda \prescript{(1\,2)}{2}N_m^A$ and  $\prescript{(2\,3)}{2}N_r^A=\prescript{(2\,3)}{3}N_r^A $. 

A quasigroup $(Q,\cdot)$ is a $\rho$-\textit{inverse quasigroup} if there exist permutations
$\rho_1,\rho_2,\rho_3$ of the set $Q$ such that 
\begin{equation} \label{inv 8}
\rho_1 (x \cdot y) \cdot \rho_2 y= \rho_3 x
\end{equation}
for all $ x,y\in Q$ \cite{Belousov&87, Keedwell&Anthony}. From definition of autostrophism, (\ref{inv 8}) is true if and only if $\theta=((1\,3),(\rho_1,\rho_2,\rho_3))$ is an autostrophism of $(Q,\cdot)$. 

From Theorem \ref{nucleus_iso}, we conclude in a $\rho$-inverse quasigroup, we have $\prescript{(1\,3)}{}N_l^A= \theta^{-1} \prescript{(1\,3)}{}N_l^A \theta$, $\prescript{(2\,3)}{}N_r^A= \theta^{-1} \prescript{(1\,2)}{}N_m^A \theta$ and $\prescript{(1\,2)}{}N_m^A= \theta^{-1} \prescript{(2\,3)}{}N_r^A \theta$, which implies 
right $(2\,3)$-A-nucleus is isomorphic to middle $(1\,2)$-A-nucleus.
Also observe that in a $\rho$-inverse quasigroup $\theta\in N(\prescript{(1\,3)}{}N_l^A)$ and $\theta^2 \in N(\prescript{(2\,3)}{}N_r^A)\cap N(\prescript{(1\,2)}{}N_m^A)$.

From Table \ref{table5} we conclude in a $\rho$-inverse quasigroup, we have $\prescript{(1\,3)}{1}N_l^A= \rho_3^{-1} \prescript{(1\,3)}{3}N_l^A \rho_1$, $\prescript{(1\,3)}{3}N_l^A= \rho_1^{-1} \prescript{(1\,3)}{1}N_l^A \rho_3$, $\prescript{(2\,3)}{2}N_r^A= \rho_3^{-1} \prescript{(1\,2)}{2}N_m^A \rho_2$, $\prescript{(2\,3)}{3}N_r^A= \rho_2^{-1} \prescript{(1\,2)}{1}N_m^A \rho_3$, $\prescript{(1\,2)}{1}N_m^A= \rho_2^{-1} \prescript{(2\,3)}{3}N_r^A \rho_1$ and $\prescript{(1\,2)}{2}N_m^A= \rho_1^{-1} \prescript{(2\,3)}{2}N_r^A \rho_2$.

A quasigroup $(Q,\cdot)$ has the \textit{right inverse property (RIP)} with respect to a permutation $\rho$ of $Q$ if 
\begin{equation} \label{inv 9}
x  y \cdot \rho y = x
\end{equation}
for all $x,y \in Q$ \cite{Belousov&67, Shcherbacov&2017}. Note that (\ref{inv 9}) holds if and only if $((1\,3),(\varepsilon,\rho,\varepsilon))$ is an autostrophism of $(Q,\cdot)$. 
Therefore, an LIP-quasigroup is a $\rho$-inverse quasigroup with $\rho_2=\rho$ and  $\rho_1=\rho_3=\varepsilon$. Also from (\ref{inv 9}), $(x y \cdot \rho y) \rho^2 y=x y$, i.e., $x \cdot \rho^2 y=x \cdot y$ which implies $\rho^2=\varepsilon$.
Hence, in an RIP-quasigroup $\prescript{(1\,3)}{1}N_l^A= \prescript{(1\,3)}{3}N_l^A$,  $\prescript{(2\,3)}{2}N_r^A=\prescript{(1\,2)}{2}N_m^A  \rho$, and $\prescript{(1\,2)}{1}N_m^A=  \rho \prescript{(2\,3)}{3}N_r^A$. 

A quasigroup $(Q,\cdot)$ has the \textit{inverse property (IP)} (with respect to permutations $\lambda$ and $\rho$ of $Q$) if it is both LIP and RIP-inverse quasigroup \cite{Belousov&67}. Thus $\lambda^2=\varepsilon$ and $\rho^2=\varepsilon$. Hence in IP-quasigroup $\prescript{(1\,3)}{1}N_l^A= \prescript{(1\,3)}{3}N_l^A$,  $\prescript{(2\,3)}{2}N_r^A=\prescript{(2\,3)}{3}N_r^A $, and $\prescript{(1\,2)}{1}N_m^A \cong \prescript{(1\,2)}{2}N_m^A$. 

A quasigroup $(Q,\cdot)$ has the \textit{WCIP (weak commutative inverse property)} with respect to a permutation $J$ of $Q$ if 
\begin{equation} \label{inv 10}
J (x \cdot y) \cdot y= J x
\end{equation}
for all $x,y \in Q$ \cite{Shcherbacov&2017}. Note that (\ref{inv 10}) is true if and only if $((1\,3),(J,\varepsilon,J))$ is an autostrophism of $(Q,\cdot)$.  
Therefore, WCIP-quasigroup is a $\rho$-inverse quasigroup with $\rho_2=\varepsilon$,  $\rho_1=\rho_3=J$ and $J^2=\varepsilon$ (from Lemma 3.15 \cite{Shcherbacov&2017}). Hence in the WCIP-quasigroup, $\prescript{(1\,3)}{1}N_l^A= J \prescript{(1\,3)}{3}N_l^A J$ or $\prescript{(1\,3)}{1}N_l^A \cong \prescript{(1\,3)}{3}N_l^A $,  $\prescript{(2\,3)}{2}N_r^A=J \prescript{(1\,2)}{2}N_m^A $, and $\prescript{(1\,2)}{1}N_m^A=  \prescript{(2\,3)}{3}N_r^A J$. 

A quasigroup $(Q,\cdot)$ is a $\mu$-\textit{inverse quasigroup} if there exist permutations
$\mu_1,\mu_2,\mu_3$ of the set $Q$ such that 
\begin{equation}\label{inv 11}
\mu_1 y \cdot \mu_2 x= \mu_3 (x \cdot y)
\end{equation}
for all $ x,y\in Q$ \cite{Shcherbacov&2017}. It may be noted that (\ref{inv 11}) holds if and only if $((1\,2),(\mu_1,\mu_2,\mu_3)) \in Aus(Q,\cdot)$. 

From Theorem \ref{nucleus_iso}, we have in a $\mu$-inverse quasigroup $\prescript{(1\,3)}{}N_l^A= \theta^{-1} \prescript{(2\,3)}{}N_r^A \theta$, $\prescript{(2\,3)}{}N_r^A= \theta^{-1} \prescript{(1\,3)}{}N_l^A \theta$ and $\prescript{(1\,2)}{}N_m^A= \theta^{-1} \prescript{(1\,2)}{}N_m^A \theta$, which implies 
left $(1\,3)$-A-nucleus is isomorphic to right $(2\,3)$-A-nucleus.
Also observe that, $\theta^2 \in N(\prescript{(1\,3)}{}N_l^A) \cap N(\prescript{(2\,3)}{}N_r^A)$ and $\theta \in N(\prescript{(1\,2)}{}N_m^A)$.

From Table \ref{table5}, we conclude that in a $\mu$-inverse quasigroup, we have $\prescript{(1\,3)}{1}N_l^A= \mu_3^{-1} \prescript{(2\,3)}{2}N_r^A \mu_1$, $\prescript{(1\,3)}{3}N_l^A= \mu_1^{-1} \prescript{(2\,3)}{3}N_r^A \mu_3$, $\prescript{(2\,3)}{2}N_r^A= \mu_3^{-1} \prescript{(1\,3)}{1}N_l^A \mu_2$, $\prescript{(2\,3)}{3}N_r^A= \mu_2^{-1} \prescript{(1\,3)}{3}N_l^A \mu_3$, $\prescript{(1\,2)}{1}N_m^A= \mu_2^{-1} \prescript{(1\,2)}{2}N_m^A \mu_1$, and $\prescript{(1\,2)}{2}N_m^A= \mu_1^{-1} \prescript{(1\,2)}{1}N_m^A \mu_2$.

\section{Conclusions}
In this paper, the characterization of the inverse sets of $\sigma$-A-nuclei and products of $\sigma$-A-nuclei and $\tau$-A-nuclei of a quasigroup, and their respective component sets have been discussed (where $\sigma, \tau \in S_3$).
Further, connections between the components of the $\sigma$-A-nuclei of a quasigroup and components of the $\sigma$-A-nuclei of its isostrophic images have been derived. Finally, we have investigated properties of $\sigma$-A-nuclei of various inverse quasigroups using the derived connections.

This study may further be carried out to obtain more relationships on $\sigma$-A-nuclei of a quasigroup. 
Also, the properties of the $\sigma$-A-nuclei may further be explored for various other quasigroup classes available in the literature.

\section*{Acknowledgement}
The research of the first author is supported by University Grants Commission (UGC), with reference no.: 20/12/2015(ii)EU-V.


\end{document}